\documentclass[10pt]{amsart}

\usepackage{amsmath,amssymb,amsthm,amsfonts,amstext}
\usepackage{tikz-cd}
\usepackage{dsfont}
\usepackage{enumitem,verbatim}
\usepackage[utf8]{inputenc}
\usepackage[english]{babel}
\usepackage[enableskew]{youngtab}
\usepackage{ytableau}
\usepackage{setspace}
\usepackage{dsfont}
\usepackage{bm}
\usepackage{tikz}
\usepackage{fullpage}
\usepackage{accents}
\usepackage[normalem]{ulem}
\usepackage{parskip}
\usepackage{upgreek}
\usepackage[font=small,labelfont=bf]{caption}
\usepackage{ragged2e}
\usepackage{hyperref}
\hypersetup{colorlinks=true}

\calclayout

\newtheorem{teo}{Theorem}[section]
\newtheorem{prop}[teo]{Proposition}
\newtheorem{cor}[teo]{Corollary}
\newtheorem{lemma}[teo]{Lemma}

\newtheorem{alg}[teo]{Algorithm}
\theoremstyle{definition}
\newtheorem{de}[teo]{Definition}
\newtheorem{rem}[teo]{Remark}

\newtheorem{ex}[teo]{Example}

\newtheorem*{teo*}{Theorem}
 
\def\a{\alpha}
\def\aa{\mathbf{a}}

\def\b{\beta}
\def\bb{\mathbf{b}}
\def\B{\mathcal{B}}

\def\e{\varepsilon}
\def\EE{\mathbf{E}}
\def\EEE{\mathcal{E}}

\def\g{\gamma}
\def\K{\mathcal{K}}
\def\l{\lambda}
\def\ll{\overline{\lambda}}
\def\f{\mathfrak{f}}
\def\m{\mathfrak{m}}

\def\Q{\mathbf{Q}}
\def\s{\sigma}
\def\t{\mathbf{t}}
\def\u{\mu}
\def\v{\nu}
\def\uu{\mathbf{u}}
\def\vv{\mathbf{v}}
\def\ww{\mathbf{w}}
\def\1{\mathbf{1}}

\def\w{\mathbf{w}}
\def\x{\mathbf{x}}
\def\y{\mathbf{y}}
\def\z{\mathbf{z}}
\def\sym{\operatorname{sym}}

\def\trop{\operatorname{trop}}
\def\conv{\operatorname{conv}}

\def\cone{\operatorname{cone}}
\def\tcone{\operatorname{tcone}}
\def\vspan{\operatorname{span}}

\def\Hom{\operatorname{Hom}}
\def\supp{\operatorname{supp}}

\def\B{\mathcal{B}}
\def\H{\mathbf{H}}
\def\HH{\mathcal{H}}

\def\F{\mathcal{F}}
\def\L{\mathbf{L}}
\def\M{\mathbf{M}}
\def\N{\mathbb{N}}
\def\P{\mathcal{P}}
\def\PP{\mathfrak{P}}
\def\R{\mathbb{R}}
\def\RR{\mathbf{R}}
\def\C{\mathbb{C}}
\def\s{\mathbf{s}}
\def\S{\Sigma}
\def\SS{\mathfrak{S}}
\def\SSS{\mathcal{S}}
\def\cS{\mathcal{S}}
\def\Sc{\mathbb{S}}
\def\cT{\mathcal{T}}

\def\rank{\operatorname{rank}}

\title{Power mean inequalities and sums of squares}
\author{Jose Acevedo}
\address{School of Mathematics, Georgia Institute of Technology, 686 Cherry Street Atlanta, GA 30332, USA}
\email{jacevedo@gatech.edu}

\author{Grigoriy Blekherman}
\address{School of Mathematics, Georgia Institute of Technology, 686 Cherry Street Atlanta, GA 30332, USA}
\email{greg@math.gatech.edu}

\thanks{The authors were partially supported by NSF grant DMS-1901950.}

\begin{document}
\maketitle 

\begin{abstract}
For fixed degree and increasing number of variables the dimension of the vector space of $n$-variate real symmetric homogeneous polynomials (forms) of degree $d$ stabilizes. We study the limits of the cones of symmetric nonnegative polynomials and symmetric sums of squares, when expressed in power-mean or monomial-mean basis. These limits correspond to forms with stable expression in power-mean (or monomial-mean) polynomials that are globally nonnegative (resp. sums of squares) regardless of the number of variables. We introduce partial symmetry reduction to describe the limit cone of symmetric sums of squares, and reprove a result of \cite{blekherman2020symmetric} that limits of symmetric nonnegative polynomials and sums of squares agree in degree $4$.  We use \emph{tropicalization} of the dual cones, which was first in the context of comparing nonnegative polynomials and sums of squares in \cite{blekherman2020tropicalization}, to show differences between cones of symmetric polynomials and sums of squares starting in degree $6$, which disproves a conjecture of \cite{blekherman2020symmetric}. For even symmetric nonnegative forms and sums of squares we show that the cones agree for degree at most $8$, and are different starting with degree 10.
We also find, via tropicalization, explicit examples of symmetric forms that are nonnegative but not sums of squares in the limit.
\end{abstract}


\section{Introduction}

Inequalities in symmetric polynomials are a classical subject going back to Newton and McLaurin.  Newton's inequalities state that $$\tilde{e}^2_k \geq \tilde{e}_{k-1}\tilde{e}_{k+1},$$ and Maclaurin's inequalities state that $$\tilde{e}_{k-1}^k\geq \tilde{e}_k^{k-1},$$ where $\tilde{e}_k$ is the $k$-th elementary symmetric mean $$\tilde{e}_k=\frac{\sum_{A\subset [n],\\ |A|=k}x^A}{\binom{n}{k}}.$$ 
Perhaps the best known inequality is the arithmetic mean-geometric mean inequality.
A thorough review of some classical inequalities is given in \cite{cuttler2011inequalities}. Notice that the above inequalities can be viewed as being independent of the number of variables, and in fact the number of variables was suppressed in the notation above. It is natural to ask what polynomial inequalities hold \emph{regardless of the number of variables}. 

Unfortunately, the answer to the above question depends on the way we identify symmetric polynomials in $n$ variables with symmetric polynomials in $n+1$ variables. For a thorough discussion we refer to \cite{blekherman2020symmetric}. It is also of interest to understand which symmetric nonnegative inequalities can be proved via a \emph{sum-of-squares decomposition}.

It is well-known that a symmetric polynomial can be expressed as a polynomial in power sum or power mean polynomials. We work with homogenenous polynomials (forms) which we express in the power-mean basis to address the question \emph{what forms with stable expression in power mean polynomials are globally nonnegative regardless of the number of variables, and how do they compare to forms that can be certified via sums of squares}? 

The relationship between nonnegative forms and sums of squares is a classical topic in real algebraic geometry. We will use $n$ to denote the number of variables, and $2d$ denote the degree. Hilbert showed that except for three cases, $n=2$, $2d=2$ and $n=3$, $2d=4$, there exist globally nonnegative forms that are not sums of squares. Equality between nonnegative forms and sums of squares transfers to symmetric forms, and as shown by Choi and Lam \cite{choi1977old}, and later Goel, Kuhlmann and Reznick \cite{goel2016choi}, there are no additional cases of equality.

However, it was shown in \cite{blekherman2020symmetric} that in degree 4 normalized symmetric forms that are nonnegative in any number of variables are sums of squares. The authors conjectured that this holds for any even degree $2d$. We disprove this conjecture for all $2d\geq 6$. For even forms (forms with all even exponents)  there is an additional case of equality $n=3$, $2d=8$ discovered by Harris \cite{harris1999real}.  Goel, Kuhlmann and Reznick showed that this is the only additional case of equality in \cite{goel2017analogue}. We show that for even forms of degree 8, normalized symmetric forms that are nonnegative in any number of variables are sums of squares, and equality fails for all higher degrees.

Our main advancement in analyzing sums of squares in arbitrary number of variables is a simplified framework via \emph{partial symmetry reduction}. An important outcome of our results is that it is actually simpler to understand symmetric polynomials that are sums of squares in any number of variables, than understanding symmetric sum of squares in a fixed number of variables $n$. There is an analogy between our framework and the framework used in the sum of squares method in extremal graph theory \cite{blekherman2020simple}. One can regard the partial symmetry reduction framework we develop as the case of partially labeled "graphs" where every edge contains only one vertex (see Section 2 of \cite{blekherman2020simple} for details). Our framework gives a quick proof of equality in the case $2d=4$, which was the main result of \cite{blekherman2020symmetric}, and the case of equality for even forms of degrees $6$ and $8$. 

To prove non-equality for $2d>4$ we consider \emph{tropicalization} of the cones dual to nonnegative forms and sums of squares. Tropicalization is a technique frequently used in complex algebraic geometry, but it has been recently shown to be useful in distinguishing between nonnegative polynomials and sums of squares \cite{blekherman2020tropicalization,blekherman2022moments}. Here we apply it to the limit cones, which describe nonnegative and sums of squares forms independent of the number of variables.

Our partial symmetrization framework allows us to find tropicalization of the cone dual to sums of squares using ideas similar to \cite{blekherman2020tropicalization}. We compare it to the tropicalization of the cone dual to nonnegative forms, whose tropicalization is computed analogously to \cite{blekherman2022moments}. Using tropicalization we also construct explicit examples of forms nonnegative for any number of variables, that are not sums of squares.

\subsection{Main Results in Detail}
\sloppy An $n$-variate polynomial $f(x_1,\dots,x_n)$ is \emph{symmetric} if the action of the symmetric group $\mathcal{S}_n$ that permutes variables leaves $f$ unchanged, i.e., $f(x_1,\dots,x_n)=f(x_{\sigma(1)},\dots,x_{\sigma(n)})$ for all $\sigma\in\mathcal{S}_n$. 
Let $\R[x_1,\dots,x_n]$ denote the ring of polynomials with real coefficients in $n$ variables. Let $H_{n,d}$ denote the vector space of $n$-variate forms of degree $d$ in $\R[x_1,\dots,x_n]$ and let $H_{n,d}^S$ denote its subspace of symmetric forms. Let $\S_{n,2d}$ denote the cone of forms in $H_{n,2d}$ which are sums of squares (observe these are sums of squares of forms in $H_{n,d}$), and let $\P_{n,2d}$ denote the cone of nonnegative forms in $H_{n,2d}$. Similarly define the cone of symmetric sums of squares $\S_{n,2d}^S$ and the cone of nonnegative symmetric forms $\P_{n,2d}^S$ in $H_{n,2d}^S$. Since every sum of squares of real polynomials is nonnegative we have $\S^S_{n,2d}\subseteq\P^S_{n,2d}$.

A natural way of assigning a symmetric form with $n+1$ variables to a form with $n$ variables is to \emph{resymmetrize}, i.e., for a symmetric form with $n$ variables $x_1,\dots,x_n$ symmetrize it with respect to the variables $x_1,\dots,x_{n+1}$. The \emph{symmetrization} of $f\in\R[x_1,\dots,x_n]$ is defined by 
\begin{align*}
\sym_n(f)=\frac1{n!}\sum_{\sigma\in\mathcal{S}_n}\sigma(f).
\end{align*}
For $m<n$ there is a natural inclusion $\varphi_{m,n}:H_{m,d}^S\to H_{n,d}^S$ that sends $f\mapsto\sym_n(f)$. These define a directed system with direct limit $H_{\infty,d}^\varphi$.
The vector spaces $H_{n,d}^S$ have the same dimension for $n\ge d$, namely $\dim(H_{n,d}^S)=\pi(d)$ the number of partitions of $d$, so also $\dim(H_{\infty,d}^\varphi)=\pi(d)$. 

Resymmetrization sends sums of squares to sums of squares and nonnegative forms to sums of nonnegative forms, hence $\varphi_{m,n}(\S_{m,2d}^S)\subseteq\S_{n,2d}^S$ and $\varphi_{m,n}(\P_{m,2d}^S)\subseteq\P_{n,2d}^S$. An explicit choice of coordinates on $H_{n,d}$ which aligns well with the maps $\varphi_{m,n}$ is the monomial mean basis in each $H_{n,2d}^S$ for $n\ge2d$. This choice of coordinates identifies $H_{n,2d}$ with $\R^{\pi(2d)}$ (we assume the bases have the lexicographic order on partitions), and there each sequence of cones is nested and increasing. More precisely, if 
\begin{align*}
\S_{n,2d}^\varphi&:=\left\{c\in\R^{\pi(2d)}:\sum_{\l\vdash2d}c_\l m_\l^{(n)}\in\S_{n,2d}^S\right\},\\
\P_{n,2d}^\varphi&:=\left\{c\in\R^{\pi(2d)}:\sum_{\l\vdash2d}c_\l m_\l^{(n)}\in\P_{n,2d}^S\right\}
\end{align*}
then $\S_{2d,2d}^\varphi\subseteq\S_{2d+1,2d}^\varphi\subseteq\cdots$ and $\P_{2d,2d}^\varphi\subseteq\P_{2d+1,2d}^\varphi\subseteq\cdots$ are non-decreasing sequences of closed, convex, containing no lines, full-dimensional, i.e., \emph{proper}, cones in $\R^{\pi(2d)}$. Let $\SS_{2d}^\varphi$ and $\PP_{2d}^\varphi$ be the closures of the limit sets of each sequence. The nestedness of the sequences $\{\S_{n,2d}^\varphi\}_{n\ge2d}$ and $\{\P_{n,2d}^\varphi\}_{n\ge2d}$ implies that
\begin{eqnarray}
\SS_{2d}^\varphi&=&\overline{\bigcup_{n\ge2d}\S_{n,2d}^\varphi},\\
\PP_{2d}^\varphi&=&\overline{\bigcup_{n\ge2d}\P_{n,2d}^\varphi}
\end{eqnarray}

We can also coordinatize with respect to the power mean basis. We call the form $p_r^{(n)}:=\frac{x_1^r+\dots+x_n^r}n$ the $r$-\emph{th power mean} and, for $\l\vdash d$, $p_\l^{(n)}=\prod_{i=1}^{\ell(\l)}p_{\l_i}$ the $\l$-th power mean. As with monomial means, the set $\{p_\l^{(n)}\}_{\l\vdash d}$ is a basis for $H_{n,d}^S$ when $n\ge d$. Sometimes we will avoid the superindex and write $p_\l$ instead of $p_\l^{(n)}$, and the same for monomial means when the set of variables is clear from the context. The power sums, and power means, satisfy the identity $p_\l p_\u=p_{\l\u}$ for any fixed number of variables. In the limit, monomial means satisfy the same identity (see Lemma \ref{gluing}).

As in \cite{blekherman2020symmetric} we also consider the sums of squares and nonnegative cones with respect to the power mean basis:
\begin{align*}
\S_{n,2d}^\rho&:=\left\{c\in\R^{\pi(2d)}:\sum_{\l\vdash2d}c_\l p_\l^{(n)}\in\S_{n,2d}^S\right\},\\
\P_{n,2d}^\rho&:=\left\{c\in\R^{\pi(2d)}:\sum_{\l\vdash2d}c_\l p_\l^{(n)}\in\P_{n,2d}^S\right\}.
\end{align*}
The sequences $\{\S_{n,2d}^\rho\}_{n\ge2d}$ and $\{\P_{n,2d}^\rho\}_{n\ge2d}$ are partially nested non-increasing. More precisely, if $m$ divides $n$ then $\S_{m,2d}^\rho\supseteq\S_{n,2d}^\rho$ and $\P_{m,2d}^\rho\supseteq\P_{n,2d}^\rho$ (see Proposition 2.6 in \cite{blekherman2020symmetric}). The limit sets of these cones then consist of the points that belong to all the cones in the sequence, i.e.,
\begin{eqnarray}
\SS_{2d}^\rho&=&\bigcap_{n\ge2d}\S_{n,2d}^\rho,\\
\PP_{2d}^\rho&=&\bigcap_{n\ge2d}\P_{n,2d}^\rho.
\end{eqnarray}

Note that $\SS_{2d}^\rho$ and $\PP_{2d}^\rho$ can be seen as the limit sets of sums of squares and nonnegative forms with respect to the transition maps $\rho_{m,n}:H_{m,d}^S\to H_{n,d}^S$ that send power means to power means, i.e., $\rho_{m,n}(p_\l^{(m)})=p_\l^{(n)}$ for $m<n$. Both cones are closed, convex and contain no lines since they are intersections of closed convex cones that contain no lines. Interestingly, $\varphi$ and $\rho$ lead to the same cones in the limit, i.e., $\SS_{2d}^\varphi=\SS_{2d}^\rho$ and $\PP_{2d}^\varphi=\PP_{2d}^\rho$, as was shown in \cite{blekherman2020symmetric} (see also Proposition \ref{phi=rho} below). And so we may remove the superindices and simply call these cones $\SS_{2d}$ and $\PP_{2d}$. This follows from the fact that monomial means \emph{behave} like power means when the number of variables goes to infinity (see the \emph{gluing lemma}, Lemma \ref{gluing}, below).

The gluing lemmas below, together with partial symmetrization, lead to a \emph{combinatorial} description of the limit cone of sums of squares $\SS_{2d}$, or equivalently of its dual cone $\SS_{2d}^*$.

\begin{teo*}[Theorem \ref{TeoPartialSym}]
\begin{align*}
\SS_{2d}^*=\{\z\in\R^{\pi(2d)}:\M_{2d}\text{ is positive semidefinite} \},
\end{align*}
where $\M_{2d}$ is a matrix whose entries are indexed by pairs $((\a,\l),(\a',\l'))$ with $\a,\a'\in\N^d$, $\l,\l'$ partitions of size at most $d$, such that $|\a|+|\l|=|\a'|+|\l'|=d$, and the $((\a,\l),(\a',\l'))$ entry of $\M_{2d}$ is a formal variable $z_{(\a+\a')\l\l'}$ where $(\a+\a')\l\l'$ denotes the partition whose parts are precisely the nonzero coordinates of $\a+\a'$ together with the parts of $\l$ and $\l'$.
\end{teo*}

\begin{ex}
For $d=2$ the complete list of $(\alpha, \lambda)$ pairs is: 
\begin{align*}
(\a,\l)\in\{((0,0),(2)),((0,0),(1,1)),((1,0),(1)),((0,1),(1)),((1,1),\emptyset),((2,0),\emptyset),((0,2),\emptyset)\},
\end{align*}
and hence 
\begin{align*}
\M_4=\begin{bmatrix}
z_{(2,2)} & z_{(2,1,1)} & z_{(2,1,1)} & z_{(2,1,1)} & z_{(2,1,1)} & z_{(2,2)} & z_{2,2}\\
z_{(2,1,1)} & z_{(1,1,1,1)} & z_{(1,1,1,1)} & z_{(1,1,1,1)} & z_{(1,1,1,1)} & z_{(2,1,1)} & z_{(2,1,1)}\\
z_{(2,1,1)} & z_{(1,1,1,1)} & z_{(2,1,1)} & z_{(1,1,1,1)} & z_{(2,1,1)} & z_{(3,1)} & z_{(2,1,1)}\\
z_{(2,1,1)} & z_{(1,1,1,1)} & z_{(1,1,1,1)} & z_{(2,1,1)} & z_{(2,1,1)} & z_{(2,1,1)} & z_{(3,1)}\\
z_{(2,1,1)} & z_{(1,1,1,1)} & z_{(2,1,1)} & z_{(2,1,1)} & z_{(2,2)} & z_{(3,1)} & z_{(3,1)}\\
z_{(2,2)} & z_{(2,1,1)} & z_{(3,1)} & z_{(2,1,1)} & z_{(3,1)} & z_{(4)} & z_{(2,2)}\\
z_{(2,2)} & z_{(2,1,1)} & z_{(2,1,1)} & z_{(3,1)} & z_{(3,1)} & z_{(2,2)} & z_{(4)}
\end{bmatrix}
\end{align*}

Therefore $\SS_4^*=\{(z_{(1,1,1,1)},z_{(2,1,1)},z_{(2,2)},z_{(3,1)},z_{(4)})\in\R^5:\M_4\succeq0\}$. 
\end{ex}

For the limit cone of even symmetric sums of squares $\EEE\SS_{2d}$ (see Section \ref{evensection}) and its dual $\EEE\SS_{2d}^*$ we have an analogous result.

\begin{teo*}[Theorem \ref{ThmEvenPartialSym}]
\begin{align*}
\EEE\SS_{2d}^*=\{\z\in\R^{\pi(d)}:\M_{2d}'\text{ is positive semidefinite} \},
\end{align*}
where $\M_{2d}'$ is a matrix whose entries are indexed by the pairs $((\a,\l),(\a',\l'))$ with $\a,\a'\in\N^d$, $\l,\l'$ even partitions of size at most $d$, $|\a|+|\l|=|\a'|+|\l'|=d$, and its $((\a,\l),(\a',\l'))$ entry is the formal variable $z_{(\a+\a')\l\l'}$ if $\a+\a'$ has all even entries, or $0$ otherwise.
\end{teo*}

\begin{ex}
For $d=2$ the complete list of $(\alpha, \lambda)$ pairs is: 
\begin{align*}
(\a,\l)\in\{((0,0),(2)),((1,1),\emptyset),((2,0),\emptyset),((0,2),\emptyset)\},
\end{align*}
and therefore 
\begin{align*}
\M_4'=\begin{bmatrix}
z_{(2,2)} & 0 & z_{(2,2)} & z_{(2,2)}\\
0 & z_{(2,2)} & 0 & 0\\
z_{(2,2)} & 0 & z_{(4)} & z_{(2,2)}\\
z_{(2,2)} & 0 & z_{(2,2)} & z_{(4)}
\end{bmatrix}
\end{align*}

Therefore $\EEE\SS_4^*=\{(z_{(2,2)},z_{(4)})\in\R^2:\M_4'\succeq0\}$.
\end{ex}

\begin{rem}\label{labeling}
The technique of partial symmetrization was first introduced in extremal graph theory, where it was used to analyze the power of sums of squares calculus in the context of proving homomorphism density inequalities in \cite{blekherman2020tropicalization}. The intuitive idea of the connection is as follows: in the case of polynomials we can view $(\alpha,\lambda)$ pairs as specifying a ``partially labelled loop graph" with $d$ vertices and $d$ edges, where each ``edge" is a loop (i.e. it contains only one vertex) in the following way. We start with vertex labelled $1, \dots, d$ and the  entry $\alpha_i$ specifies the number of loops/edges containing vertex $i$ (if the entry is 0, then vertex $i$ is omitted).
We add the number of part of $\lambda$-many unlabelled vertices, the number of edges containing each vertex given by the size of the corresponding part of $\lambda$. 
\end{rem}

\begin{ex} The pair $(\alpha, \lambda)=\{(0,1),(1)\}$ corresponds to the following partially labelled loop graph:
\begin{tikzpicture}[shorten >=1pt,auto,node distance=1cm,
                thick,main node/.style={circle,draw,font=\normalsize}]
                 \node[main node] (1) {2};
  \node[main node] (2) [right of=1] {};
\end{tikzpicture},
 and the pair $(\alpha, \lambda)=\{(1,0),(1)\}$ to the following graph:
\begin{tikzpicture}[shorten >=1pt,auto,node distance=1cm,
                thick,main node/.style={circle,draw,font=\normalsize}]
                 \node[main node] (1) {1};
  \node[main node] (1) [right of=1] {};
\end{tikzpicture}.
\end{ex}

We show that, in the limit, nonnegative symmetric forms of degree $2d$ are not equal to sums of squares for all $2d\geq 6$. The case of equality for $2d=4$ was shown in \cite{blekherman2020symmetric}, and we also provide a different proof of this equality.
\begin{teo}\label{Thm1}
$\SS_{2d}\subsetneq\PP_{2d}$ if and only if $2d\ge6$.
\end{teo}

We also prove the corresponding result for even symmetric forms, where we show that in the limit nonnegative even symmetric forms agree with sums of squares for the case of $2d=8$ (which was asked in \cite{debus2020reflection}), and are strictly larger for $2d\geq 10$.

\begin{teo}\label{Thm2Even}
$\EEE\SS_{2d}\subsetneq\EEE\PP_{2d}$ if and only if $2d\ge10$.
\end{teo}

We use the technique of \emph{tropicalizaton} to distinguish between the limit cones. While this is a well-established technique in algebraic geometry, its application to the study of the relationship between nonnegative polynomials and sums of squares is quite new, and we follow some ideas of \cite{blekherman2022moments}.


\subsection*{Acknowledgements}
We thank Sebastian Debus and Cordian Riener for many fruitful conversations and helping us figure out explicit constructions of limit nonnegative forms that are not sums of squares. We thank Josephine Yu for helpful discussions regarding tropicalization and providing us with Algorithm \ref{Algorithm}.

\section{Normalized symmetric forms in any number of variables}\label{SectionSSOS}



A \emph{partition} $\l=(\l_1,\dots,\l_k)$ of $d$, denoted $\l\vdash d$, is a non-increasing sequence of positive integers $\l_1\ge\dots\ge\l_k$ that add up to $d$, we call these integers the \emph{parts} of $\l$. The \emph{length} of $\l$, denoted $\ell(\l)$, is its number of parts. The \emph{size} of $\l$, denoted $|\l|$, is the sum of its parts. We denote by $\Lambda$ the set of partitions of all natural numbers (the empty set $\emptyset$ the only partition of $0$ by definition, and so $|\emptyset|=\ell(\emptyset)=0$). We will condense the number of repeated parts of a partition in a exponent, for example $(2^4,1^2):=(2,2,2,2,1,1)$.

A \emph{monomial mean} is the symmetrization of a monomial, i.e., the arithmetic mean of the terms in a monomial symmetric polynomial. For all $n>m\ge d$ the maps $\varphi_{m,n}$ take the basis of monomial means in $H_{m,d}^S$ to the basis of monomial means in $H_{n,d}^S$. If $M_\l^{(n)}$ is the monomial symmetric polynomial in $n$-variables corresponding to the partition $\l$ then we call $m_\l^{(n)}$ its arithmetic mean. Namely, $m_\l^{(n)}:=M_\l^{(n)}/M_\l^{(n)}(\1)$ where $\1$ is the all-ones vector. Sometimes we write $m_\l(\y)$ for the monomial mean corresponding to the partition $\l$ in the set of variables $\y$, for example $m_{(2,1)}(x_2,x_3,x_4)=\frac16(x_2^2x_3+x_2x_3^2+x_2^2x_4+x_2x_4^2
+x_3^2x_4+x_3x_4^2)$. We will often remove the parentheses in the subindex of $m_\l$, for example $m_{2,1}:=m_{(2,1)}$. Notice that $M_\l$ and $m_\l$ are only defined when the number of variables is at least the length of $\l$, so we assume that $M_\l^{(n)}=m_\l^{(n)}=1$ whenever $\ell(\l)>n$.

For $\l,\u\in\Lambda$ define $\l\u$ to be the partition whose parts are precisely the parts of $\l$ and $\u$ put together, we say that we \emph{glue} $\l$ and $\u$. Note $\ell(\l\u)=\ell(\l)+\ell(\u)$ and $|\l\u|=|\l|+|\u|$. We will also label the coordinates of $\R^{\pi(d)}$ according to the lexicographic order on partitions, e.g., given $\mathbf{c}\in\R^{\pi(3)}$ then $\mathbf{c}:=(c_{(1^3)},c_{(2,1)},c_{(3)})$.

For a symmetric form $f$ and $\v\vdash\deg(f)$ define $[f]_\v$ to be the coefficient of $m_\v$ in $f$ when expressed in the monomial mean basis.

\begin{lemma}\label{gluing}
For all $\l,\u,\v\in\Lambda$
$$\lim_{n\to\infty}[m_\l^{(n)}m_\u^{(n)}]_\v=\begin{cases}
0, & \text{if}\,\,\,\nu\ne\l\u,\\
1, & \text{if}\,\,\,\nu=\l\u.
\end{cases}$$
\end{lemma}

\begin{ex}
\begin{align*}
m_{1^2}^{(n)}m_{1}^{(n)}&=\frac{\sum_{i<j} x_ix_j}{\binom{n}2}\cdot\frac{\sum x_k}{n}=\frac1{n\binom{n}2}\left(\sum_{i\ne j} x_i^2x_j+3\sum_{i<j<k} x_ix_jx_k\right)\\
&=\frac2n\cdot\frac{\sum_{i\ne j} x_i^2x_j}{2\binom{n}2}+\frac{n-2}n\cdot\frac{\sum_{i<j<k} x_ix_jx_k}{\binom{n}3}\\
&=\frac2{n}m_{2,1}^{(n)}+\frac{n-2}{n}m_{1^3}^{(n)},\qquad\frac2n\to0\quad\text{and}\quad\frac{n-2}n\to1.
\end{align*}{\flushright$\triangle$}
\end{ex}

Proof of Lemma \ref{gluing}
\begin{proof}
Let $m_\l^{(n)}m_\u^{(n)}=\sum_{\v\vdash r} c_\v^{(n)}m_\v^{(n)}$, where $c_\v^{(n)}\in\R$ and $r=|\l|+|\u|$. Clearly, the $c_\v^{(n)}$ are nonnegative. It is enough to prove that $\sum_{\v\ne\l\u}c_\v^{(n)}\to0$ since $\sum_{\v\vdash r}c_\v^{(n)}=1$ (this can be seen by evaluating at $\1$).
For any partition $\rho$ define $t(\rho):=\prod_i a_i!$ where $a_i$ is the number of parts of $\rho$ equal to $i$.
Now suppose $n$ is sufficiently large. The number of monomials in $m_\rho^{(n)}$ is $O(n^{\ell(\rho)})$ because we have $\binom{n}{\ell(\rho)}\sim\frac{n^{\ell(\rho)}}{\ell(\rho)!}$ ways of choosing the variables and $\frac{\ell(\rho)!}{t(\rho)}$ ways of assigning exponents to those variables, so the number of monomials is $\binom{n}{\ell(\rho)}\frac{\ell(\rho)!}{t(\rho)}\sim\frac{n^{\ell(\rho)}}{t(\rho)}=O(n^{\ell(\rho)})$.
The number of monomials in $m_\l^{(n)}m_\u^{(n)}$ is therefore $O(n^{\ell(\l)})O(n^{\ell(\u)})=O(n^{\ell(\l\u)})$. The number of monomials in $m_\l^{(n)}m_\u^{(n)}$ of shape $\v\ne\l\u$ can be counted in the following way: a monomial $\x^\a$ in $m_\l^{(n)}$ can be chosen in $O(n^{\ell(\l)})$ ways and a monomial $\x^\b$ in $m_\u^{(n)}$ with at least one variable in common with $\x^\a$ (so that $\x^\a \x^\b$ has shape different from $\l\u$) can be chosen in $O(n^{\ell(\u)-1})$ ways: from the total number of monomials in $m_\u^{(n)}$ exclude the ones without variables in common with $\x^\a$:
\begin{align}\label{binom}
\left(\binom{n}{\ell(\u)}-\binom{n-\ell(\l)}{\ell(\u)}\right)\frac{\ell(\u)!}{t(\u)}\sim\frac1{t(\u)}(n^{\ell(\u)}-(n-\ell(\l))^{\ell(\u)})\sim\frac{\ell(\l)\ell(\u)}{t(\u)}n^{\ell(\u)-1}.
\end{align}
So, if $M_\l^{(n)}M_\u^{(n)}=\sum_{\v\vdash r} d_\v^{(n)}M_\v^{(n)}$ then $c_\v^{(n)}=d_\v^{(n)}\frac{M_\v^{(n)}(\1)}{M_\l^{(n)}(\1)M_\u^{(n)}(\1)}=d_\v^{(n)}\frac{M_\v^{(n)}(\1)}{O(n^{\ell(\l\u)})}$. But, for $\v\ne\l\u$ (such that $d_\nu^{(n)}\ne0$), $M_\v^{(n)}(\1)$ is at most $O(n^{\ell(\l\u)-1})$, so $c_\v^{(n)}$ is at most $d_\v^{(n)}O(n^{-1})$. Finally, $d_\v^{(n)}$ is $O(1)$ since the number of ways of obtaining a fixed monomial $\x^\g$ in the product $M_\l^{(n)}M_\u^{(n)}$ does not depend on $n$ ($|\g|=r$ so $\x^\g$ has at most $r$ fixed variables), so $c_\v^{(n)}\to0$ as $n\to\infty$ and therefore $\sum_{\v\ne\l\u}c_\v\to0$ since the number of partitions of $r$ is also $O(1)$.
\end{proof}

\begin{de}
For $\l\vdash d$ and $m<n$ one has $\varphi_{m,n}(m_\l^{(m)})=m_\l^{(n)}$, so one can define $\m_\l\in H_{\infty,d}^\varphi$ as the corresponding direct limit of $m_\l^{(n)}$, i.e., the preimage of $m_\l^{(n)}$ under the natural isomorphism $\varphi_n:H_{\infty,d}^\varphi\to H_{n,d}^S$, which is well-defined since the natural isomorphisms commute with the transition maps, i.e., $\varphi_n=\varphi_{m,n}\circ\varphi_m$ or $\varphi_m^{-1}=\varphi_n^{-1}\circ\varphi_{m,n}$ for all $m<n$ with $m,n\ge d$. 
\end{de}

\begin{de}\label{philimit}
We define the $\varphi$-\emph{limit} (and analogously the $\rho$-limit) of a sequence $$\left\{f_n=\sum_{\l\vdash d}c_\l^{(n)}m_\l^{(n)}\quad\bigg|\quad c_\l^{(n)}\in\R\right\}_{n\ge d}$$ of symmetric forms in increasing number of variables, denoted $\lim_{n\to\infty}^\varphi f_n$, as $\sum_{\l\vdash d}(\lim_{n\to\infty}c_\l^{(n)})\m_\l$. 
\end{de} 

\subsection{Normalized symmetric nonnegative forms and symmetric sums of squares at infinity}\label{subSection2}

By identifying an expression $\f=\sum_{\l\vdash2d} c_\l\m_\l$ with its tuple of coefficients $c\in\R^{\pi(2d)}$, we will abuse notation and say that $\f\in\SS_{2d}$ or $\f\in\PP_{2d}$ whenever $c\in\SS_{2d}$ or $c\in\PP_{2d}$ respectively. Given $\f=\sum_{\l\vdash2d} c_\l\m_\l$ define, for $n\ge 2d$, $\f^{\varphi,(n)}:=\sum_{\l\vdash2d}c_\l m_\l^{(n)}$ and $\f^{\rho,(n)}:=\sum_{\l\vdash2d} c_\l p_\l^{(n)}$.

\begin{prop}\label{onceNEQalwaysNEQ}
If for some $d$ we have $\SS_{2d}\subsetneq\PP_{2d}$ then $\SS_{2k}\subsetneq\PP_{2k}$ for all $k\ge d$.
\end{prop}
\begin{proof}
Suppose $\f=\sum_{\l\vdash2d}c_\l \m_\l\in\PP_{2d}\setminus\SS_{2d}$. So for all $n\ge2d$ we have $\f^{\rho,(n)}\in \P_{n,2d}^S$ but $\f^{\rho,(n)}\notin\Sigma_{n,2d}^S$ for all $n$ in some infinite set $I$. Then for all $n\ge2d$ we have $\f^{\rho,(n)}(p_1^{(n)})^2\in \P_{n,2d+2}^S$ but, since $p_1^{(n)}$ is irreducible indefinite for all $n$, $\f^{\rho,(n)}(p_1^{(n)})^2\notin\Sigma_{n,2d+2}^S$ for all $n\in I$, so $\sum_{\l\vdash2d}c_\l\m_{\l 1^2}\in\PP_{2d+2}\setminus\SS_{2d+2}$.
\end{proof}

In Section 6 we will prove that $\SS_6\subsetneq\PP_6$ (and $\EEE\SS_{2d}\subsetneq\EEE\PP_{2d}$ for $2d=10,12,14,16$), and so Proposition \ref{onceNEQalwaysNEQ} will immediately imply Theorem \ref{Thm1}. Theorem \ref{Thm2Even} will follow from Proposition \ref{Prop5678} and Proposition \ref{PropEven} below, but first we show that there exist even symmetric irreducible indefinite forms of degree 4 with fixed coefficients regardless of the number of variables.

\begin{lemma}\label{LemmaIrreducible}
There exist $a,b\in\R$ such that for all $n\ge3$ the even symmetric form $g_n:=ap_4^{(n)}+b(p_2^{(n)})^2$ is irreducible indefinite.
\end{lemma}
\begin{proof}
In the space of $n$-variate forms of degree $d$ the coefficients of a reducible polynomial satisfy the equations coming from one of finitely many possible factorizations. Hence the set of reducible polynomials in this space is Zariski closed. Therefore, when restricting to a subspace one only needs to see that there is at least one irreducible form there to prove that the set of irreducible forms is Zariski open in that subspace. Since there is an irreducible even symmetric form of degree four \cite{goel2017analogue} for each $n\ge3$, the set of $n$-variate even symmetric irreducible forms of degree four is Zariski open in that subspace, so the intersection of these sets for each $n\ge3$ is Zariski open too because we are excluding only countably many hypersurfaces. Finally, the set of indefinite forms is Euclidean open in the subspace of even symmetric forms of degree 4 so it has non-empty intersection with the Zariski open set above.


\end{proof}

\begin{prop}\label{PropEven}
If $\EEE\SS_{2k}\subsetneq\EEE\PP_{2k}$ for $k=d,d+1,d+2,d+3$ then $\EEE\SS_{2k}\subsetneq\EEE\PP_{2k}$ for all $k\ge d$.
\end{prop}
\begin{proof}
Suppose $\f=\sum_{\l\vdash2d}c_\l \m_\l\in\EEE\PP_{2d}\setminus\EEE\SS_{2d}$. So for all $n\ge2d$ we have $\f^{\rho,(n)}\in \P_{n,2d}^\rho$ but $\f^{\rho,(n)}\notin\Sigma_{n,2d}^\rho$ for all $n$ in some infinite set $I$. Now, taking $g_n$ as in Lemma \ref{LemmaIrreducible}, for all $n\ge2d$ we have $\f^{\rho,(n)}g_n^2\in \P_{n,2d+8}^\rho$ but, since $g_n$ is irreducible indefinite for all $n$, $\f^{\rho,(n)}g_n^2\notin\Sigma_{n,2d+8}^\rho$ for all $n\in I$, so $\lim_{n\to\infty}^\rho \f^{\rho,(n)}g_n^2 \in\EEE\PP_{2d+8}\setminus\EEE\SS_{2d+8}$ (this limit exists since the coefficients of $g_n$ do not depend on $n$ by Lemma \ref{LemmaIrreducible}).
\end{proof}

\section{Symmetry reduction and the limit pseudomoment cone}\label{SectionPseudomomentCone}
A form $f\in H_{n,d}$ is a sum of squares if and only if there exists a positive semidefinite matrix $A$ (denoted $A\succeq0$) such that $f=\langle A,\uu\uu^\top\rangle$, where $\uu$ is a vector whose entries are the monomials of degree $d$ in the variables $x_1,\dots,x_n$ and $\langle\cdot,\cdot\rangle$ denotes the Frobenius inner product. Similarly, $f\in H_{n,d}^S$ is a sum of squares if and only if there exists $A\succeq0$ such that $f=\langle A,\sym_n(\uu\uu^\top)\rangle$. The row size of the matrix $\sym_n(\uu\uu^\top)$ equals the dimension of $H_{n,d}$ as a real vector space, which is $O(n^d)$. However, by using a \emph{symmetry basis} for $H_{n,d}$ \cite{gatermann2004symmetry, blekherman2020symmetric} one can find a block diagonal matrix $\Q_{n,2d}=\sym_n(\vv\vv^\top)$ of size depending only on $d$, where $\vv$ is a vector whose entries are the elements of this symmetry basis, so that $f\in H_{n,d}^S$ is a sum of squares if and only if there exists $A\succeq0$ such that $f=\langle A,\Q_{n,2d}\rangle$. Moreover, a stability property (Proposition \ref{decomp}) of the decomposition of $H_{n,d}$ into irreducible $\mathcal{S}_n$-submodules guarantees that the blocks of $\Q_{n,2d}$ remain the same size for all $n\ge2d$ and furthermore, via a gluing lemma (Lemma \ref{basis}), that there exists a $\varphi$-limit $\Q_{2d}$ when $n$ goes to infinity. This will allow us to give a membership criterion similar to the above ones, i.e., $\f\in\SS_{2d}$ if and only if there exists $A\succeq0$ such that $\f=\langle A,\Q_{2d}\rangle$ (Proposition \ref{Qrep}, which is similar to Theorem 4.15 in \cite{blekherman2020symmetric}). The blocks of $\Q_{2d}$ correspond to the $\mathcal{S}_n$-isotypic components of $H_{n,d}$ and their row size is the number of $\mathcal{S}_n$-irreducible copies in the isotypic component (in the examples below we will see that not every isotypic component is needed, further reducing the size of $\Q_{2d}$).

\subsection{Symmetry bases}

We construct the blocks of the matrices $\Q_{2d}$ with \emph{symmetry bases}. 

\begin{de}\label{symbasis}
Let $G$ be a finite group and $V$ a finite dimensional $G$-module. A \emph{symmetry basis} of $V$ is a set $\{u_1,\dots,u_r\}\subset V$ such that 
\begin{enumerate}
\item for each $1\le i\le r$, the $G$-orbit of $u_i$ spans an irreducible $G$-module $U_i$,
\item $V=\bigoplus_{i=1}^r U_i$,
\item if $\Hom_G(U_i,U_j)\ne\{0\}$ then there exists a $G$-isomorphism from $U_i$ to $U_j$ that sends $u_i$ to $u_j$.
\end{enumerate}
\end{de}

For the purposes of this paper $G=\mathcal{S}_n$, whose irreducible representations have been widely studied \cite{fulton2013representation, sagan2013symmetric, serre1977linear}. The irreducible representations of $\mathcal{S}_n$ are in bijective correspondence with the partitions $\l$ of $n$, these are the \emph{Specht modules} $\Sc_\l$, which are subspaces of the group algebra $\C[\mathcal{S}_n]$ and which can be defined via Young symmetrizers. The \emph{Young diagram} of a partition $\l$ is the left justified array of boxes whose first row, from top to bottom, has $\l_1$ boxes, whose second row has $\l_2$ boxes, and so on.
\begin{center}
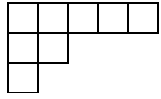
 
$\tiny\ydiagram{5,2,1}$
\captionof{figure}{The Young diagram corresponding to the partition $(5,2,1)$}
\end{center}

A \emph{Young tableau} is a Young diagram each of whose boxes contains an element from a given set. 
For $\l\vdash n$ let $\t$ be a Young tableau of shape $\l$ filled with the numbers $1,\ldots,n$. Let $R_i$ be the permutation group of the elements of row $i$ of $\t$ and $C_j$ the permutation group of the elements of column $j$ of $\t$. 
 
Define the following elements in $\C[\mathcal{S}_n]$:
\begin{eqnarray*}
\aa_\t&=&\prod_{i=1}^{\ell(\l)}\frac1{|R_i|}\sum_{\sigma\in R_i}\sigma,\\
\bb_\t&=&\prod_{j=1}^{\l_1}\frac1{|C_j|}\sum_{\sigma\in C_j}(-1)^{\sigma}\sigma\quad \text{where}\,\,\, (-1)^\sigma\,\,\,\text{denotes the sign of}\,\, \sigma,\\
Y_\t&=&\aa_\t \bb_\t.
\end{eqnarray*}
The $Y_\t$ are the \emph{Young symmetrizers}; $\aa_\t$ is the \emph{row symmetrizer} of $\t$, and $\bb_\t$ is the \emph{column symmetrizer} of $\t$. We will abuse notation and call $\aa_\t$, $\bb_\t$ and $Y_\t$ instead by the shape of $\t$, i.e., we will call them $\aa_\l$, $\bb_\l$ and $Y_\l$. The Specht modules are the images of the group algebra under Young symmetrizers, i.e., $\Sc_\l=Y_\l\C[\mathcal{S}_n]$ ($\cong_{\mathcal{S}_n}\bb_\l\aa_\l\C[\mathcal{S}_n]\cong_{\mathcal{S}_n}\C[\mathcal{S}_n]Y_\l\cong_{\mathcal{S}_n}\C[\mathcal{S}_n]\bb_\l\aa_\l$ where $\cong_{\mathcal{S}_n}$ denotes $\mathcal{S}_n$-isomorphism) where $Y_\l V:=\{Y_\l(v):v\in V\}$.

For an $\mathcal{S}_n$-module $V$ and $\l\vdash n$, denote by $V_\l$ the \emph{isotypic component} of $V$ corresponding to $\l$, i.e., $V=\bigoplus_{\l\vdash n}V_\l$ with $V_\l\cong_{\mathcal{S}_n}\Sc_\l^{\oplus v_\l}$ where $v_\l\in\N$ is the number of copies of $\Sc_\l$ in $V$. The following proposition says that Young symmetrizers can be used to find symmetry bases for the isotypic components of $\SSS_n$-modules.

The real irreducible representations of $\SSS_n$ are in fact \emph{absolutely irreducible}, i.e., remain irreducible when considered as representations over $\C$, so Schur's lemma holds.

\begin{prop}\label{copies}
Let $V$ be an $\SSS_n$-module and $\l\vdash n$. 

\begin{enumerate}
\item[(i)] The number of copies of $\Sc_\l$ in $V$ equals $\dim(Y_\l V)$.

\item[(ii)] If $v\in V$ is such that $Y_\l(v)\ne0$ then the span of the orbit of $Y_\l(v)$ is $\SSS_n$-isomorphic to $\Sc_\l$.

\item[(iii)] If $S$ spans $V$ then any maximal independent subset of $Y_\l S$ is a symmetry basis for $V_\l$. 
\end{enumerate}
\end{prop}
\begin{proof}
(i) Let $\u\vdash n$. By \cite[Exercise 4.24]{fulton2013representation} we have $Y_\l\C[\mathcal{S}_n]Y_\u=\{0\}$ when $\l\ne\u$. In such case $Y_\l\Sc_\u\cong Y_\l(\C[\mathcal{S}_n]Y_\u)=\{0\}$, hence $Y_\l V_\u=\{0\}$ and therefore $Y_\l V=Y_\l V_\l$. If $V$ has no copies of $\Sc_\l$ then $V_\l=\{0\}$ and so $\dim(Y_\l V)=0$.

If $V_\l\ne\{0\}$ then $V$ contains a copy $U\cong_{\mathcal{S}_n}\Sc_\l$. But $Y_\l\Sc_\l$ is one dimensional \cite[Lemma 4.23]{fulton2013representation}, so $Y_\l U$ will be too and therefore $Y_\l V\ne\{0\}$. Now let $V_\l=\bigoplus_{i=1}^r U_i$ with $U_i\cong_{\mathcal{S}_n}\Sc_\l$. $U_i$ is an $\mathcal{S}_n$-module, so $Y_\l U_i\subseteq U_i$, and $Y_\l U_i$ is one-dimensional, so $Y_\l U_i=\vspan(u_i)$ for some $u_i\in U_i\setminus\{0\}$. Hence $X:=\{u_1,\dots,u_r\}$ is a basis for $Y_\l V$, and so $V$ contains $\dim(Y_\l V)$ copies of $\Sc_\l$.

Note $X$ is a symmetry basis for $V_\l$: it satisfies (1) and (2) in Definition \ref{symbasis}, and it also satisfies (3): by Schur's lemma we have $\dim(\Hom_{\mathcal{S}_n}(U_i,U_j))=1$, so there exists a unique (up to scaling) $\mathcal{S}_n$-isomorphism $\varphi_{ij}$ between them. As an $\SSS_n$-isomorphism, $\varphi_{ij}$ commutes with $Y_\l$, and as shown below, it sends $u_i$ to a nonzero multiple of $u_j$. This is because for each $k=1,\dots,r$ we have $Y_\l U_k=\vspan(u_k)$, so $u_k=Y_\l(v_k)$ for some $v_k\in U_k$, and so  
\begin{align*}
    \varphi_{ij}(u_i)=\varphi_{ij}(Y_\l(v_i))=Y_\l(\varphi_{ij}(v_i))=c_{ij}u_j
\end{align*}
for some $c_{ij}\in\R$, nonzero because $u_i$ is a nonzero vector.

(ii) Let $v=w+w'$ with $w\in V_\l$ and $w'\in\bigoplus_{\u\ne\l} V_\u$. Let also $w=w_1+\dots+w_r$ with $w_i\in U_i$. Then $Y_\l(v)=Y_\l(w)=Y_\l(w_1)+\dots+Y_\l(w_r)=a_1u_1+\dots+a_ru_r$ for some $a_i\in\R$. The orbit of $Y_\l(v)$ is then $\SSS_n\cdot Y_\l(v):=\{a_1(\sigma\cdot u_1)+\dots+a_r(\sigma\cdot u_r):\sigma\in\SSS_n\}$, and let $W:=\vspan(\SSS_n\cdot Y_\l(v))$. Hence 
\begin{align*}
    Y_\l W&=\vspan(a_1Y_\l(\sigma\cdot u_1)+\dots+a_rY_\l(\sigma\cdot u_r):\sigma\in\SSS_n)\\
    &=\vspan(a_1(Y_\l\sigma)(u_1)+\dots+a_r(Y_\l\sigma)(u_r):\sigma\in\SSS_n)\\
    &=\vspan(a_1(c_\sigma u_1)+\dots+a_r(c_\sigma u_r):\sigma\in\SSS_n)\\
    &=\vspan(a_1u_1+\dots+a_ru_r)=\vspan(Y_\l(v))
\end{align*}
where $c_\sigma\in\R$ for each $\sigma\in\SSS_n$. And since $Y_\l(v)\ne0$ then $Y_\l W$ is one-dimensional (not all the $c_\sigma$ can be zero), so by (i), $W$ contains only one copy of $\Sc_\l$, but since $W\subseteq V_\l$ then $W\cong_{\SSS_n}\Sc_\l$. We remark that $(Y_\l\sigma)(u_i)=c_\sigma u_i$, for a constant $c_\sigma$ not depending on $i$, because $Y_\l\sigma=\sigma Y_{\sigma^{-1}\l}$ (see \cite[Lemma 2.3.3]{sagan2013symmetric}, also $\sigma^{-1}\l$ corresponds to the tableau obtained by permuting the entries of $\t_\l$ by $\sigma^{-1}$), the $U_i$ are $\SSS_n$-isomorphic, and $(Y_\l\sigma)(u_i)=Y_\l(\sigma\cdot u_i)\in\vspan(u_i)$.

(iii) Since $S$ spans $V$ then the $V_\l$ components of the elements of $S$ will span $V_\l$, so $Y_\l S=\{Y_\l(w_1),\dots,Y_\l(w_k)\}$ for some $w_i$ such that $V_\l=\vspan(w_1,\dots,w_k)$, so $Y_\l S$ spans $Y_\l V_\l$. Now select a maximal independent subset of $Y_\l S$, say $M:=\{Y_\l(w_1),\dots,Y_\l(w_r)\}$ without loss of generality. By (ii), for each $i=1,\dots, r$, the orbit of $Y_\l(w_i)$ spans a copy $A_i$ of $\Sc_\l$. Since $M$ is a basis for $Y_\l V_\l$, and $Y_\l V_\l=\vspan(u_1,\dots,u_r)$, then $U_j\subseteq\bigoplus_{i=1}^r A_i$ for each $j=1,\dots,r$, and since $V_\l=\bigoplus_{i=1}^r U_i$ then $V_\l=\bigoplus_{i=1}^r A_i$. The above proves that $M$ satisfies the first two conditions for being a symmetry basis for $V_\l$, and the third condition follows as in the proof in (i) for $X$.
\end{proof}

A remarkable fact is that the decomposition of $H_{n,d}$ into irreducible $\mathcal{S}_n$-submodules remains \emph{stable} when $n\ge2d$ \cite[Theorem 4.7]{riener2013exploiting}. This stability is crucial for symmetry reduction at the limit, and thus for our description of $\SS_{2d}$ and its dual. 

Let $\ll$ denote the partition obtained from the partition $\l$ by removing its first part, e.g., $\overline{(10,3,2,1)}=(3,2,1)$ or $\overline{(n)}=\emptyset$.

\begin{prop}\label{decomp}
For each positive integer $d$ and each $\mu\in\Lambda_{\le d}$ there are nonnegative integers $h_{\u,d}$ such that for all $n\ge2d$ $$H_{n,d}\cong_{\SSS_n}\bigoplus_{\u\in \Lambda_{\le d}}\Sc_{(n-|\u|,\u)}^{\oplus h_{\u,d}},$$ where $(n-|\u|,\u)$ is the partition whose first part is $n-|\u|$ and whose other parts are the parts of $\u$. 
\end{prop}

The nonzero $h_{\mu,d}$ are known as the \emph{fake degrees} of $H_{n,d}$. Call $\F_d:=\{\mu\in\Lambda_{\le d}\,|\, h_{\mu,d}\ne0\}$.

\begin{rem}\label{n-d}
For $n\ge d$, the number of copies of the trivial representation $\Sc_{(n)}$ in $H_{n,d}$ equals $\pi(d)$, the dimension of $H_{n,d}^S$, hence $h_{\emptyset,d}=\pi(d)$. For $n\ge2d$ the number of copies of $\Sc_{(n-d,d)}$ in $H_{n,d}$ is exactly one, so $h_{(d),d}=1$, and also this is the only $\Sc_{\l}$ with $\l_1=n-d$ appearing in $H_{n,d}$. To see this, notice that $\dim(H_{n,d})\sim\frac{n^d}{d!}$ and the hook length formula implies that if $\l_1=n-l$, for $n\gg l$, then $\dim(\Sc_\l)\sim\dim(\Sc_{\ll})\frac{n^l}{l!}$. So the only irreducibles $\Sc_\l$ with $\l_1=n-d$ that could appear in $H_{n,d}$ are the ones with $\ll=(d)$ or $\ll=(1^d)$ (which correspond to the only one-dimensional representations of $\mathcal{S}_d$), and exactly one must appear. $\Sc_{(n-d,1^d)}$ cannot appear since $Y_{(n-d,1^d)}$ annihilates every monomial in $H_{n,d}$.
\end{rem}

\begin{de}
For $\l\vdash n$ let $\t_\l$ be the unique Young tableau of shape $\l$ with entries $1,\dots,n$ such that for any two entries in $\t_\l$ the following conditions are satisfied:
\begin{enumerate}
\item if they are on different rows then the one in the lowest row is smaller,
\item if they are in the same row then the leftmost is smaller. 
\end{enumerate}
For example $$\t_{(n)}=\tiny\Yvcentermath1\young(123\cdot\cdot\cdot n)\, ,\,\,\t_{(n-1,1)}=\tiny\Yvcentermath1\young(2\cdot\cdot\cdot n,1)\, ,\,\,\t_{(n-5,3,2)}=\tiny\Yvcentermath1\young(678\cdot\cdot\cdot n,345,12)$$
\end{de}
From now on we will assume that $Y_\l$ is the Young symmetrizer associated to the tableau $\t_\l$.

\begin{ex}[symmetry basis for $H_{n,2}$]\label{hn2}
First we show that for all $n\ge4$
\begin{align*}
H_{n,2}\cong_{\mathcal{S}_n}(\Sc_{(n)})^{\oplus2}\oplus(\Sc_{(n-1,1)})^{\oplus2}\oplus\Sc_{(n-2,2)}.
\end{align*}

By Remark \ref{n-d} we have $h_{\emptyset,2}=2$, $h_{(1^2),2}=0$ and $h_{(2),2}=1$, and by Proposition \ref{decomp} it only remains to find $h_{(1),2}$. For this purpose we can use the hook length formula to count dimensions, noting that $\dim(H_{n,d})=\binom{n+d-1}{d}$. There is also a combinatorial way of computing these multiplicities by using \emph{charges} of tableaux \cite[Proposition 4.13]{blekherman2020symmetric}. However, we also need to compute symmetry bases, so we do so with Young symmetrizers as in Proposition \ref{copies}.
 
Let $X=\{x_1^2,\dots,x_n^2,x_1x_2,\dots,x_{n-1}x_n\}$ be a basis of monomials for $H_{n,2}$. Observe $Y_{(n)}X=\{m_2^{(n)}, m_{1^2}^{(n)}\}$ is a symmetry basis for $(H_{n,2})_{(n)}$.
 
Since $h_{(1),2}=2$, a symmetry basis for $(H_{n,2})_{(n-1,1)}$ consists of two linearly independent elements in $Y_{(n-1,1)}X$:
\begin{align*}
Y_{(n-1,1)}(x_1x_3)&=\aa_{(n-1,1)}\bb_{(n-1,1)}(x_1x_3)=\frac12\aa_{(n-1,1)}(x_1x_3-x_2x_3)=\frac12x_1m_1(\x_{[1]})-\frac12m_{1^2}(\x_{[1]}),\\
Y_{(n-1,1)}(x_1^2)&=\frac12\aa_{(n-1,1)}(x_1^2-x_2^2)=\frac12x_1^2-\frac12m_2(\x_{[1]})
\end{align*}
where $m_\l(\x_{[k]}):=m_\l(x_{k+1},\dots,x_n)$. Finally $h_{(2),2}=1$, so a symmetry basis for $(H_{n,2})_{(n-2,2)}$ consist of just 
\begin{align*}
Y_{(n-2,2)}(x_1x_2)&=\aa_{(n-2,2)}\bb_{(n-2,2)}(x_1x_2)=\aa_{(n-2,2)}\frac14(x_1x_2-x_1x_4-x_2x_3+x_3x_4)\\
&=\frac{x_1x_2}4-\frac{x_1+x_2}4 m_1(\x_{[2]})+\frac14 m_{1^2}(\x_{[2]}).
\end{align*}
Rescaling we have the following symmetry bases for the isotypic components $(H_{n,2})_{(n)}$, $(H_{n,2})_{(n-1,1)}$ and $(H_{n,2})_{(n-2,2)}$:
\begin{align*}
\B_{\emptyset,2}^{(n)}&:=\{m_{1^2}^{(n)},\, m_2^{(n)}\},\\
\B_{(1),2}^{(n)}&:=\{x_1m_1(\x_{[1]})-m_{1^2}(\x_{[1]}),\, x_1^2-m_2(\x_{[1]})\},\\
\B_{(2),2}^{(n)}&:=\{x_1x_2-(x_1+x_2)m_1(\x_{[2]})+m_{1^2}(\x_{[2]})\}
\end{align*}
A symmetry basis for $H_{n,2}$ is just the union of these sets.
$\triangle$
\end{ex}

It will be practical to denote by $\B_{\u,d}$ the set obtained after removing the superindices and the $\x_{[k]}$'s in $\B_{\u,d}^{(n)}$. For example 
$$\B_{\emptyset,2}:=\{m_{1^2},m_2\}\qquad\text{and}\qquad \B_{(1),2}:=\{x_1m_1-m_{1^2},x_1^2-m_2\}.$$
Note that from $\B_{\u,d}$ we can recover $\B_{\u,d}^{(n)}$ just by evaluating the monomial means in $\x_{[k]}$ where $\x=\{x_1,\dots,x_n\}$ and $k=|\u|$.

From a symmetry basis for $H_{n,d}$ we compute the matrices $\Q_{n,2d}$ describing $\Sigma_{n,2d}^S$. These matrices are block diagonal, with a block $\Q_{\mu,2d}^{(n)}$ for each $\mu\in\F_d$. For example, if $\B_{(1),2}^{(n)}=\{f_n,g_n\}$, then 
\begin{align*}
    \Q_{(1),4}^{(n)}=\begin{bmatrix}
    \sym_n(f_n^2) & \sym_n(f_ng_n)\\
    \sym_n(f_ng_n) & \sym_n(g_n^2).
\end{bmatrix}
\end{align*}
Now observe that the coefficients of the $x^\a m_\l$ in each element of the symmetry bases considered do not depend on the number of variables. This, together with Lemma \ref{gluing}, guarantees the existence of $\lim_{n\to\infty}^\varphi \Q_{n,2d}$, from which we describe $\SS_{2d}$. For example
\begin{align*}   \Q_{(1),4}:=\lim_{n\to\infty}\Q_{(1),4}^{(n)}=\begin{bmatrix}
    \lim_{n\to\infty}^\varphi\sym_n(f_n^2) & \lim_{n\to\infty}^\varphi\sym_n(f_ng_n)\\
    \lim_{n\to\infty}^\varphi\sym_n(f_ng_n) & \lim_{n\to\infty}^\varphi\sym_n(g_n^2).
\end{bmatrix}
\end{align*}
In place of $\lim_{n\to\infty}^\varphi\sym_n$ we will simply write $\sym_\infty$.

\begin{lemma}\label{basis}
Fix $k,d\in\N$ with $k\le d$, fix $\a,\b\in\N^d$ such that $\supp(\a),\supp(\b)\subseteq\{1,\dots,k\}$, and fix $\l,\u\in\Lambda$. Let $\x=(x_1,\dots,x_n)$. Then 
\begin{align*}
\sym_\infty(\x^\alpha m_\l(\x_{[k]})\cdot \x^\b m_\u(\x_{[k]}))=\m_{\pi}
\end{align*}
where $\pi$ is the partition obtained by gluing $\a+\b$, $\l$ and $\u$.
\end{lemma}
\begin{proof}
Using the notation of Lemma \ref{gluing}, we have
\begin{eqnarray*}
\sym_n(\x^{\a+\b}m_\l(\x_{[k]})m_\u(\x_{[k]}))&=&\sym_n\sum_\v c_\v^{(n)}\x^{\a+\b}m_\v(\x_{[k]})\\
&=&\sum c_\v^{(n)}\sym_n(\x^{\a+\b}m_\v(\x_{[k]}))\\
&=&\sum c_\v^{(n)} m_{(\a+\b)\v}^{(n)}\to \m_{(\a+\b)\l\u}
\end{eqnarray*}
where the third equality follows since $\supp(\a),\supp(\b)\subseteq\{1,\dots,k\}$ and the last step follows from Lemma \ref{gluing}.
\end{proof}

\begin{de}
Given the symmetry basis $\B_{\u,d}^{(n)}=\{f_1^{(n)},\dots,f_{h_{\u,d}}^{(n)}\}$ for $(H_{n,d})_\l$ where $\l=(n-|\u|,\u)$ is such that $\u\in\F_d$, we define the \emph{isotypic matrix} of $(H_{n,d})_\l$ with respect to $\B_{\u,d}^{(n)}$ as the $h_{\u,d}\times h_{\u,d}$ matrix $\Q_{\u,2d}^{(n)}:=\left(\sym_n(f_i^{(n)}f_j^{(n)})\right)_{i,j}$. 

Define $\Q_{\u,2d}:=\lim^\varphi_{n\to\infty}\Q_{\u,2d}^{(n)}$, which exists by Lemma \ref{basis} (Proposition \ref{prop4} provides an example).
\end{de}


Let $S_{+}^k$ denote the cone of positive semidefinite $k\times k$ matrices. The following theorem is implied by Corollary 4.4 in \cite{blekherman2020symmetric}, along with the corresponding statement for the limit cone.
\begin{teo}\label{Qrep}
For all $n\ge2d$
\begin{align*}
\S_{n,2d}^S&=\left\{\sum_{\u\in \F_d}\langle A_\u,\Q_{\u,2d}^{(n)}\rangle: A_\u\in S^{h_{\u,d}}_+\text{ for each }\u\in \F_d\right\},\\
\SS_{2d}&=\left\{\sum_{\u\in \F_d}\left\langle A_\u,\Q_{\u,2d}\right\rangle: A_\u\in S^{h_{\u,d}}_+\text{ for each }\u\in \F_d\right\}.
\end{align*} 
\end{teo}
\begin{rem}
Notice that Theorem \ref{Qrep} gives a way to test, via semidefinite programming, whether a given symmetric form or symmetric limit form is a sum of squares.
\end{rem}

\begin{prop}\label{prop4}
The limit sums of squares cone $\SS_4$ is given by 
\begin{align*}
\Q_{\emptyset,4}=\begin{bmatrix}
\m_{1^4} & \m_{2,1^2}\\
\m_{2,1^2} & \m_{2^2}
\end{bmatrix}\quad\text{and}\quad\Q_{(1),4}=\begin{bmatrix}
\m_{2,1^2}-\m_{1^4} & \m_{3,1}-\m_{2,1^2}\\
\m_{3,1}-\m_{2,1^2} & \m_{4}-\m_{2^2}
\end{bmatrix}.
\end{align*}
Namely,
\begin{align*}
\SS_4=\{\langle A,\Q_{\emptyset,4}\rangle+\langle B,\Q_{(1),4}\rangle\,|\, A,B\in S^2_+\}.
\end{align*}
\end{prop}
\begin{proof}
Using Lemma \ref{basis} on the bases obtained in Example \ref{hn2} we obtain
\begin{align*}
\sym_\infty(m_{1^2}\cdot m_{1^2})=\m_{1^4},\,\,\,
\sym_\infty(m_{1^2}\cdot m_{2})&=\m_{2,1^2},\\
\sym_\infty(m_{2}\cdot m_{2})&=\m_{2^2},
\end{align*}
therefore $\Q_{\emptyset,4}=\begin{bmatrix}
\m_{1^4} & \m_{2,1^2}\\
\m_{2,1^2} & \m_{2^2}
\end{bmatrix}$. 

From the symmetry basis $\B_{(1),2}=\{x_1m_1-m_{1^2},\, x_1^2-m_2\}$ we get $(x_1m_1-m_{1^2})^2=x_1^2m_1^2-2x_1m_1m_{1^2}+m_{1^2}^2$ and then $$\sym_\infty(x_1^2m_1^2-2x_1m_1m_{1^2}+m_{1^2}^2)=\sym_\infty(x_1^2m_1^2)-2\sym_\infty(x_1m_1m_{1^2})+\sym_\infty(m_{1^2}^2)$$ which by Lemma \ref{basis} equals $\m_{2,1^2}-2\m_{1^4}+\m_{1^4}=\m_{2,1^2}-\m_{1^4}$. Similarly,
\begin{align*}
\sym_\infty((x_1m_1-m_{1^2})(x_1^2-m_2))=\m_{3,1}-\m_{2,1^2}\text{ and }
\sym_\infty(x_1^2-m_2)^2=\m_{4}-\m_{2^2},
\end{align*}
therefore $\Q_{(1),4}=\begin{bmatrix}
\m_{2,1^2}-\m_{1^4} & \m_{3,1}-\m_{2,1^2}\\
\m_{3,1}-\m_{2,1^2} & \m_{4}-\m_{2^2}
\end{bmatrix}$. From the symmetry basis $\B_{(2),2}$ we get
$$\Q_{(2),4}=\sym_\infty(x_1x_2-(x_1+x_2)m_1+m_{1^2})^2=\m_{2^2}-2\m_{2,1^2}+\m_{1^4}.$$
Hence $\SS_4=\{\langle A_\emptyset,\Q_{\emptyset,4}\rangle+\langle A_{(1)},\Q_{(1),4}\rangle+\langle A_{(2)},\Q_{(2),4}\rangle:A_\emptyset\in S_2^+,A_{(1)}\in S_2^+,A_{(2)}\ge0\}$. However, one can see that in fact 
\begin{align*}
\SS_4=\{\langle A_\emptyset,\Q_{\emptyset,4}\rangle+\langle A_{(1)},\Q_{(1),4}\rangle:A_\emptyset\in S_2^+,A_{(1)}\in S_2^+\}
\end{align*}
since $\Q_{(2),4}=(1,-1)\Q_{\emptyset,4}(1,-1)^\top$.
\end{proof}

\begin{prop}\label{prop6}
The limit sums of squares cone $\SS_6$ is given by
\begin{align*}
\Q_{\emptyset,6}&=\begin{bmatrix}
\m_{1^6} & \m_{2,1^4} & \m_{3,1^3}\\ \m_{2,1^4} & \m_{2^2,1^2} & \m_{3,2,1}\\ \m_{3,1^3} & \m_{3,2,1} & \m_{3^2}
\end{bmatrix},\\
\Q_{(1),6}&=\begin{bmatrix}
\m_{2,1^4}-\m_{1^6} & \m_{3,1^3}-\m_{2,1^4} & \m_{2^2,1^2}-\m_{2,1^4} & \m_{4,1^2}-\m_{3,1^3}\\
\m_{3,1^3}-\m_{2,1^4} & \m_{4,1^2}-\m_{2^2,1^2} & \m_{3,2,1}-\m_{2^2,1^2} & \m_{5,1}-\m_{3,2,1}\\
\m_{2^2,1^2}-\m_{2,1^4} & \m_{3,2,1}-\m_{2^2,1^2} & \m_{2^3}-\m_{2^2,1^2} & \m_{4,2}-\m_{3,2,1} \\
\m_{4,1^2}-\m_{3,1^3} & \m_{5,1}-\m_{3,2,1} & \m_{4,2}-\m_{3,2,1} & \m_6-\m_{3^2}
\end{bmatrix}\quad\text{and}\\
\Q_{(1^2),6}&=\m_{4,2}-\m_{4,1^2}-\m_{3^2}+2\m_{3,2,1}-\m_{2^3}.
\end{align*}
Namely,
\begin{align*}
\SS_6=\{\langle A,\Q_{\emptyset,6}\rangle+\langle B,\Q_{(1),6}\rangle+\alpha\Q_{(1^2),6}\,|\,A\in S^3_+,B\in S^4_+,\alpha\ge0\}.
\end{align*}
\end{prop}
\begin{proof}
See Appendix (\ref{appiso6}).
\end{proof}

\subsection{Partial symmetry reduction}\label{ssectionPartialSymR}

Symmetry bases are not straightforward to compute in general. However one can bypass finding them explicitly and instead consider the possible terms that could appear in them, and this will be enough for describing the limit sum of squares cone. 
Let $\cS$ be a finite set of polynomials and denote by $\sum\cS^2$ the set of sums of squares of polynomials in the linear span of $\cS$. Observe that if for finite sets of polynomials $\cS,\cT$ one has $\cS\subseteq\vspan(\cT)$, then $\sum\cS^2\subseteq\sum\cT^2$. Equivalently, if $\{v_1,\dots,v_k\}\subseteq\vspan\{w_1,\dots,w_l\}$ where the $v_i$ and $w_j$ are polynomials, then 
\begin{align*}
\{\langle A,\vv^\top\vv\rangle:A\in S^k_+\}\subseteq\{\langle B,\w^\top\w\rangle:B\in S^l_+\}
\end{align*}
where $\vv:=(v_1,\dots,v_k)$ and $\w:=(w_1,\dots,w_l)$. Using this idea and the fact that the terms of the symmetry basis of $H_{n,d}$ constructed above are of the form $\x^\a m_\v$ where $\supp(\x^\a)\subseteq\{1,2,\dots,k\}$ for some $k\le d$, then, by setting $\w$ to be a vector whose entries are the $\x^\a m_\v$, Theorem \ref{Qrep} implies that $\SS_{2d}=\{\langle A,\RR_{2d}\rangle:A\succeq0\}$ where $\RR_{2d}:=\sym_\infty(\w^\top\w)$. Moreover, the entries of $\RR_{2d}$ are monomial means $\m_\pi$, and the row size of $\RR_{2d}$ only depends on $d$. To see this we only need a small improvement to Lemma \ref{basis}.

\begin{lemma}\label{basis2}
Fix $k,l,d\in\N$ with $k,l\le d$, fix $\a,\b\in\N^d$ such that $\supp(\a)\subseteq\{1,\dots,k\}$ and $\supp(\b)\subseteq\{1,\dots,l\}$, and fix $\l,\u\in\Lambda$. Let $\x=(x_1,\dots,x_n)$. Then 
\begin{align*}
   \sym_\infty(\x^\alpha m_\l(\x_{[k]})\cdot\x^\b m_\u(\x_{[l]}))=\m_{\pi} 
\end{align*}
where $\pi$ is the partition obtained by gluing $\a+\b$, $\l$ and $\u$.
\end{lemma}
\begin{proof}
If $k=l$ the result is precisely Lemma \ref{basis}, so suppose $k>l$ and write $M_\u(\x_{[l]})=M_\u(\x_{[k]})+\sum \x^\u$ where each monomial in the sum $\sum \x^\u$ has at least one variable in $\{x_{l+1},\dots,x_k\}$. Then
\begin{align*}
\x^\a m_\l(\x_{[k]})\cdot \x^\b m_\u(\x_{[l]})=\x^{\a+\b}m_\l(\x_{[k]})m_\u(\x_{[k]})\frac{\binom{n-k}{\ell(\mu)}}{\binom{n-l}{\ell(\mu)}}+\x^{\a+\b}m_\l(\x_{[k]})\frac1{\binom{n-l}{\ell(\u)}\frac{\ell(\u)!}{t(\u)}}\sum \x^\u.
\end{align*}
As in (\ref{binom}) in the proof of Lemma \ref{gluing}, the number of monomials in $\sum \x^\u$ is $O(n^{\ell(\u)-1})$, but $\binom{n-l}{\ell(\u)}\frac{\ell(\u)!}{t(\u)}$ is $O(n^{\ell(\u)})$, hence the coefficients of $\sym_n(\x^{\a+\b}m_\l(\x_{[k]})\frac1{\binom{n-l}{\ell(\u)}\frac{\ell(\u)!}{t(\u)}}\sum \x^\u)$ in the monomial mean basis go to zero as $n$ goes to infinity. And then we proceed as in Lemma \ref{basis}.
\end{proof}

\begin{prop}\label{proppartial4}
\begin{align*}
\SS_4&=\{\langle A,\RR_4\rangle\,|\, A\in S^4_+\}\quad\text{where}\\
\RR_4&=\begin{bmatrix}\m_{1^4} & \m_{1^4} & \m_{2,1^2} & \m_{2,1^2}\\ \m_{1^4} & \m_{2,1^2} & \m_{2,1^2} & \m_{3,1}\\ \m_{2,1^2} & \m_{2,1^2} & \m_{2^2} & \m_{2^2}\\ \m_{2,1^2} & \m_{3,1} & \m_{2^2} & \m_4\end{bmatrix}.
\end{align*}
\end{prop}
\begin{proof}
By Proposition \ref{prop4} we only need $\Q_{\emptyset,4}$ and $\Q_{(1),4}$ to describe $\SS_4$, thus collecting the terms in $\B_{\emptyset,2}^{(n)}$ and $\B_{(1),2}^{(n)}$ (see Example \ref{hn2}) we get $$\vv:=(m_{1^2}^{(n)},m_2^{(n)},x_1m_1(\x_{[1]}),m_{1^2}(\x_{[1]}),x_1^2,m_2(\x_{[1]})).$$
Observe we can always take them with positive coefficients since the linear span does not change. However, in light of Lemma \ref{basis2}, we see that there are redundant entries in $\vv$ which will give equal rows in $\sym_\infty(\vv^\top\vv)$. For example, if $f=m_{1^2}^{(n)}$ and $g=m_{1^2}(\x_{[1]})$ then $\sym_\infty(fv_i)=\sym_\infty(gv_i)$ for any entry $v_i$ of $\vv$. These are redundancies in the sense that, after removing them, the set $\{\langle A,\sym_\infty(\vv^\top\vv)\rangle:A\succeq0\}$ will not change. After removing redundancies we obtain
$$\ww:=(m_{1^2},x_1m_1,m_2,x_1^2),$$
and then $\SS_4=\{\langle A,\RR_4\rangle:A\in S^4_+\}$ where 
\begin{align*}
\RR_4:=\sym_\infty(\ww^\top\ww)=\begin{bmatrix}
\m_{1^4} & \m_{1^4} & \m_{2,1^2} & \m_{2,1^2}\\
\m_{1^4} & \m_{2,1^2} & \m_{2,1^2} & \m_{3,1}\\
\m_{2,1^2} & \m_{2,1^2} & \m_{2^2} & \m_{2^2}\\
\m_{2,1^2} & \m_{3,1} & \m_{2^2} & \m_4
\end{bmatrix}.
\end{align*}
\end{proof}

We define $\RR_{2d}$ in an analogous way for higher degrees, although not in a precise way. A straightforward way to construct an $\RR_{2d}$ would be to collect all the terms in a symmetry basis for $H_{n,d}$, remove redundant terms modulo Lemma \ref{basis2} to get a row vector of terms $\ww$, and then compute $\RR_{2d}:=\sym_\infty(\ww^\top\ww)$. If one also knows that only some of the $\B_{\mu,d}$ are needed, like in Proposition \ref{proppartial4}, then one can compute a simpler $\RR_{2d}$. We will call any of the matrices obtained in such a way a \emph{partial symmetrization matrix for} $\SS_{2d}$. The $\RR_{2d}$ look simpler and can be computed with less effort than the $\Q_{2d}$ matrices at the cost of being possibly larger. With them we will compute the tropicalization of the dual cone to sums of squares (see Section \ref{sectiontrop}).

\begin{teo}\label{Rrep}
Let $\RR_{2d}$ be an $N\times N$ partial symmetrization matrix for $\SS_{2d}$, then
\begin{align*}
\SS_{2d}&=\left\{\left\langle A,\RR_{2d}\right\rangle\,|\, A\in S^N_+\right\}.
\end{align*} 
\end{teo}

\begin{prop}\label{proppartial6}
\begin{align*}
\SS_6&=\{\langle A,\RR_6\rangle\,|\, A\in S^{11}_+\}\quad\text{where}\\
\RR_6&=\begin{bmatrix}
\m_{1^6} & \m_{1^6} & \m_{2,1^4} & \m_{2,1^4} & \m_{2,1^4} & \m_{2,1^4} & \m_{2,1^4} & \m_{2,1^4} & \m_{2,1^4} & \m_{3,1^3} & \m_{3,1^3}\\
\m_{1^6} & \m_{2,1^4} & \m_{2,1^4} & \m_{2,1^4} & \m_{2^2,1^2} & \m_{2,1^4} & \m_{3,1^3} & \m_{2^2,1^2} & \m_{3,1^3} & \m_{3,1^3} & \m_{4,1^2}\\
\m_{2,1^4} & \m_{2,1^4} & \m_{2^2,1^2} & \m_{2^2,1^2} & \m_{2^2,1^2} & \m_{2^2,1^2} & \m_{2^2,1^2} & \m_{2^2,1^2} & \m_{2^2,1^2} & \m_{3,2,1} & \m_{3,2,1}\\
\m_{2,1^4} & \m_{2,1^4} & \m_{2^2,1^2} & \m_{2^3} & \m_{2^2,1^2} & \m_{3,2,1} & \m_{2^2,1^2} & \m_{3,2,1} & \m_{2^3} & \m_{3,2,1} & \m_{3,2,1}\\
\m_{2,1^4} & \m_{2^2,1^2} & \m_{2^2,1^2} & \m_{2^2,1^2} & \m_{2^3} & \m_{2^2,1^2} & \m_{3,2,1} & \m_{2^3} & \m_{3,2,1} & \m_{3,2,1} & \m_{4,2}\\
\m_{2,1^4} & \m_{2,1^4} & \m_{2^2,1^2} & \m_{3,2,1} & \m_{2^2,1^2} & \m_{4,1^2} & \m_{2^2,1^2} & \m_{4,1^2} & \m_{3,2,1} & \m_{3,2,1} & \m_{3,2,1}\\
\m_{2,1^4} & \m_{3,1^3} & \m_{2^2,1^2} & \m_{2^2,1^2} & \m_{3,2,1} & \m_{2^2,1^2} & \m_{4,1^2} & \m_{3,2,1} & \m_{4,1^2} & \m_{3,2,1} & \m_{5,1}\\
\m_{2,1^4} & \m_{2^2,1^2} & \m_{2^2,1^2} & \m_{3,2,1} & \m_{2^3} & \m_{4,1^2} & \m_{3,2,1} & \m_{4,2} & \m_{3^2} & \m_{3,2,1} & \m_{4,2}\\
\m_{2,1^4} & \m_{3,1^3} & \m_{2^2,1^2} & \m_{2^3} & \m_{3,2,1} & \m_{3,2,1} & \m_{4,1^2} & \m_{3^2} & \m_{4,2} & \m_{3,2,1} & \m_{5,1}\\
\m_{3,1^3} & \m_{3,1^3} & \m_{3,2,1} & \m_{3,2,1} & \m_{3,2,1} & \m_{3,2,1} & \m_{3,2,1} & \m_{3,2,1} & \m_{3,2,1} & \m_{3^2} & \m_{3^2}\\
\m_{3,1^3} & \m_{4,1^2} & \m_{3,2,1} & \m_{3,2,1} & \m_{4,2} & \m_{3,2,1} & \m_{5,1} & \m_{4,2} & \m_{5,1} & \m_{3^2} & \m_{6}
\end{bmatrix}.
\end{align*}
\end{prop}
\begin{proof}
We collect the terms $$m_{1^3}, x_1m_{1^2}, m_{2,1}, x_2m_2, x_1m_2, x_2^2m_1, x_1^2m_1,  x_1x_2^2, x_1^2x_2, m_{3}, x_1^3$$ from the symmetry bases $\B_{\emptyset,3}$, $\B_{(1),3}$ and $\B_{(1^2),3}$ (see Example \ref{apphn3}), arrange them in a row vector $\vv$, and compute $\sym_\infty(\vv^\top\vv)$.
\end{proof}

\subsection{Even symmetric case}\label{evensection}
A polynomial is called \emph{even} if all of its monomials are squares. Denote by $X^e$ the set of even polynomials in a set $X$ of polynomials. Let $\e:\R[x_1,\dots,x_n]\to\R[x_1,\dots,x_n]^e$ be the linear projection that removes non-square monomials, i.e., $\e(\x^\a)=0$ if $\a_i$ is odd for some $i$, and $\e(\x^\a)=\x^\a$ otherwise. For $f\in\R[x_1,\dots,x_n]$ we have
\begin{align}\label{signs}
\e(f)=\frac1{2^n}\sum_{\epsilon\in\{-1,1\}^n}f(\epsilon_1x_1,\dots,\epsilon_nx_n).
\end{align}
In $(H_{n,2d}^S)^e$ we have the cone of even symmetric sums of squares $\EEE\S_{n,2d}^S$ and the cone of even symmetric nonnegative forms $\EEE\P_{n,2d}^S$. Notice that $(H_{n,2d}^S)^e$ has dimension $\pi(d)$ (for $n\ge2d$) and monomial means $m_\l$ and power means $p_\l$, where $\l$ runs over partitions of $2d$ whose parts are all even, constitute bases for it.

From equation (\ref{signs}) it follows that $\EEE\S_{n,2d}^S=\e(\S_{n,2d}^S)$ and $\EEE\P_{n,2d}^S=\e(\P_{n,2d}^S)$. As in Section 2 define similarly $\EEE\S_{n,2d}^{\varphi}$, $\EEE\S_{n,2d}^{\rho}$, $\EEE\P_{n,2d}^{\varphi}$ and $\EEE\P_{n,2d}^{\rho}$ (now subsets of $\R^{\pi(d)}$). All results of Section \ref{SectionSSOS} hold analogously for this setting so similarly define their $\varphi$-limits $\EEE\SS_{2d}$ and $\EEE\PP_{2d}$.

For a matrix $\M$ whose entries are polynomials denote $\e(\M)$ to be the matrix where $\e$ is applied to each entry of $\M$. Then Theorem \ref{Qrep} implies that 
\begin{align*}
\EEE\S_{n,2d}^S=\e(\S_{n,2d}^S)=\left\{\sum_{\u\in \F_d}\langle A_\u,\e(\Q_{\u,2d}^{(n)})\rangle: A_\u\in S^{h_{\u,d}}_+\text{ for each }\u\in \F_d\right\}.
\end{align*}
It is not clear that computing $\e(\Q_{\u,2d}^{(n)})$ is easier than computing $\Q_{\u,2d}^{(n)}$ and afterwards removing non-square monomials. However, the story is different at the limit since, by Lemma \ref{basis}, in $\sym_n(\x^\a m_\l\cdot \x^\b m_\u)$ only the coefficient of the partition obtained by gluing $\a+\b,\l,\u$ does not go to zero when $n$ goes to infinity. Thus, if any of $\a+\b$, $\l$ or $\u$ has an odd part then $\e(\sym_n(\x^\a m_\l\cdot \x^\b m_\u))$ will go to zero. Hence, if we want to compute $\e(\Q_{\u,2d})$ without computing $\Q_{\u,2d}$ first, we can remove every term in $\B_{\u,d}$ which has the form $\x^\a m_\l$ where $\l$ has at least one odd part. 
\begin{de}
Given an ordered symmetry basis $\B_{\u,d}=\{f_1,\dots,f_{h_{\u,d}}\}$ let $I_\u$ be the ordered set of indices $i$ such that $\e(f_i)\ne0$, and let $h_{\u,d}^e=|I_\u|$. Define the $h_{\u,d}^e\times h_{\u,d}^e$ matrix $$\EE_{\u,2d}:=\left(\sym_\infty(\e(f_i)\e(f_j))\right)_{i,j}$$ where $i,j\in I_\u$. 
\end{de}
The following is a consequence of the above discussion.
\begin{cor}\label{evensos}
\begin{align*}
\EEE\SS_{2d}&=\left\{\sum_{\u\in \F_d}\langle A_\u,\e(\Q_{\u,2d})\rangle: A_\u\in S^{h_{\u,d}}_+\text{ for each }\u\in \F_d\right\}\\
&=\left\{\sum_{\u\in \F_d}\langle A_\u,\EE_{\u,2d}\rangle: A_\u\in S^{h_{\u,d}^e}_+\text{ for each }\u\in \F_d\right\}
\end{align*}
where $\e(\m_\l)=0$ if $\l$ has an odd part and $\e(\m_\l)=\m_\l$ otherwise.
\end{cor}

Analogously to $\RR_{2d}$ we define partial symmetrization matrices $\RR_{2d}^e$ obtained from non-redundant terms from the symmetry bases $\B_{\u,d}$ without considering terms $\x^\a m_\l$ where $\l$ has at least one odd part. 

\begin{cor}\label{evenPartialSOS}
$$\EEE\SS_{2d}=\{\langle A,\RR_{2d}^e\rangle:A\succeq0\}.$$
\end{cor}

\begin{prop}\label{propE6}
The limit sums of squares cone $\SS_6^e$ is given by 
\begin{align*}
\EE_{(1),6}=\begin{bmatrix}\m_{2^3} & \m_{4,2}\\ \m_{4,2} & \m_6\end{bmatrix} \quad\text{and}\quad\EE_{(1^2),6}=\m_{4,2}-\m_{2^3}.
\end{align*}
Namely,
\begin{align*}
\EEE\SS_6=\{\langle A,\EE_{(1),6}\rangle + \a\EE_{(1^2),6}\,|\, A\in S^2_+,\a\ge0\}.
\end{align*}
Also, by partial symmetry reduction
\begin{align*}
\EEE\SS_6=\left\{\left\langle A,\begin{bmatrix}
\m_{2^3}&\m_{2^3}&\m_{4,2}\\
\m_{2^3}&\m_{4,2}&\m_{4,2}\\
\m_{4,2}&\m_{4,2}&\m_6
\end{bmatrix}\right\rangle:A\in S^3_+\right\}.
\end{align*}
\end{prop}
\begin{proof}
See Appendix (\ref{appes6}).
\end{proof}

\begin{prop}\label{propE8}
The limit sums of squares cone $\EEE\SS_8$ is given by
\begin{align*}
\EE_{\emptyset,8}=\begin{bmatrix}\m_{2^4}&\m_{4,2^2}\\
\m_{4,2^2}&\m_{4^2}\end{bmatrix},\quad
\EE_{(1),8}=\begin{bmatrix}\m_{4,2^2}-\m_{2^4}&\m_{6,2}-\m_{4,2^2}\\ 
\m_{6,2}-\m_{4,2^2}&\m_8-\m_{4^2}\end{bmatrix}\quad\text{and}\quad
\EE_{(1^2),8}&=\m_{6,2}-\m_{4^2}.
\end{align*}
Namely,
\begin{align*}
\EEE\SS_8=\{\langle A,\EE_{\emptyset,8}\rangle+\langle B,\EE_{(1),8}\rangle+\a\EE_{(1^2),8}:A,B\in S^2_+,\alpha\ge0\}.
\end{align*}
\end{prop}
\begin{proof}
See Appendix (\ref{appes8}).
\end{proof}

We will use the partial symmetrization description in the following example to show, in Section \ref{sectionApp}, that $\EEE\SS_{10}\subsetneq\EEE\PP_{10}$.

\begin{prop}\label{propE10}
The limit sums of square cone $\EEE\SS_{10}$ is given by
\begin{align*}
\EE_{(1),10}&=\begin{bmatrix}
\m_{2^5}&\m_{4,2^3}&\m_{4,2^3}&\m_{6,2^2}\\
\m_{4,2^3}&\m_{4^2,2}&\m_{4^2,2}&\m_{6,4}\\
\m_{4,2^3}&\m_{4^2,2}&\m_{6,2^2}&\m_{8,2}\\
\m_{6,2^2}&\m_{6,4}&\m_{8,2}&\m_{10}
\end{bmatrix}\quad\text{and}\\
\EE_{(1^2),10}&=\begin{bmatrix}
\m_{4,2^3}-\m_{2^5}&\m_{4^2,2}-\m_{4,2^3}&\m_{6,2^2}-\m_{4,2^3}\\
\m_{4^2,2}-\m_{4,2^3}&\m_{6,4}-\m_{6,2^2}&\m_{6,4}-\m_{4^2,2}\\
\m_{6,2^2}-\m_{4,2^3}&\m_{6,4}-\m_{4^2,2}&\m_{8,2}-\m_{4^2,2}
\end{bmatrix}.
\end{align*}
Namely,
\begin{align*}
\EEE\SS_{10}=\{\langle A,\EE_{(1),10}\rangle+\langle B,\EE_{(1^2),10}\rangle\,\mid\, A\in S^4_+, B\in S^3_+\}.
\end{align*}
Also, by partial symmetry reduction
\begin{align*}
\EEE\SS_{10}&=\{\langle A,\RR_{10}^e\rangle\,|\, A\in S^7_+\}\quad\text{where}\\
\RR_{10}^e&=\begin{bmatrix}
\m_{2^5} & \m_{2^5} & \m_{4,2^3} & \m_{4,2^3} & \m_{4,2^3} & \m_{4,2^3} & \m_{6,2^2}\\
\m_{2^5} & \m_{4,2^3} & \m_{4,2^3} & \m_{4,2^3} & \m_{4^2,2} & \m_{6,2^2} & \m_{6,2^2}\\
\m_{4,2^3} & \m_{4,2^3} & \m_{4^2,2} & \m_{4^2,2} & \m_{4^2,2} & \m_{4^2,2} & \m_{6,4}\\
\m_{4,2^3} & \m_{4,2^3} & \m_{4^2,2} & \m_{6,2^2} & \m_{6,2^2} & \m_{4^2,2} & \m_{8,2}\\
\m_{4,2^3} & \m_{4^2,2} & \m_{4^2,2} & \m_{6,2^2} & \m_{6,4} & \m_{6,4} & \m_{8,2}\\
\m_{4,2^3} & \m_{6,2^2} & \m_{4^2,2} & \m_{4^2,2} & \m_{6,4} & \m_{8,2} & \m_{6,4}\\
\m_{6,2^2} & \m_{6,2^2} & \m_{6,4} & \m_{8,2} & \m_{8,2} & \m_{6,4} & \m_{10}
\end{bmatrix}.
\end{align*}
\end{prop}
\begin{proof}
See Appendix (\ref{appes10}).
\end{proof}

\subsection{Duality}
From the descriptions given above for limit sums of squares cones one automatically obtains descriptions for their duals. 



For example, by Proposition \ref{propE6} it follows that 
\begin{align*}
\EEE\SS_6^*&=\{(x,y,z)\in\R^3:\begin{bmatrix}x & y\\ y & z\end{bmatrix}\succeq0,\, y-x\ge0\}\\
&=\{(x,y,z)\in\R^3:\begin{bmatrix}x & x & y\\ x & y & y\\ y & y & z\end{bmatrix}\succeq0\}.
\end{align*}
So each of the above descriptions of the limit sums of squares cones automatically give a description of their duals. In Section \ref{sectionMC} we study the cones $\PP_{2d}^*$ and $\EEE\PP_{2d}^*$, which we call \emph{moment cones}. We call the cones $\SS_{2d}^*$ and $\EEE\SS_{2d}^*$ \emph{pseudomoment cones}.

From partial symmetrization we get the following results, which give descriptions of the pseudomoment cones and the even pseudomoment cones from matrices which can be sistematically constructed by \emph{gluing} (as in Remark \ref{labeling} and Lemma \ref{basis2}).

\begin{teo}\label{TeoPartialSym}
\begin{align*}
\SS_{2d}^*=\{\z\in\R^{\pi(2d)}:\M_{2d}\text{ is positive semidefinite} \},
\end{align*}
where $\M_{2d}$ is a matrix whose entries are indexed by the pairs $((\a,\l),(\a',\l'))$ with $\a,\a'\in\N^d$, $\l,\l'$ partitions of size at most $d$, such that $|\a|+|\l|=|\a'|+|\l'|=d$, and its $((\a,\l),(\a',\l'))$ entry is the variable $z_{(\a+\a')\l\l'}$ where $(\a+\a')\l\l'$ denotes the partition whose parts are precisely the nonzero coordinates of $\a+\a'$ together with the parts of $\l$ and $\l'$.
\end{teo}
\begin{proof}
The terms appearing in the symmetry basis $\B_{\u,d}$ are of the form $\x^\a m_\l$ where $|\a|+|\l|=d$. Observe that the matrix $\M_{2d}$ arises from $\sym_\infty(\vv^\top\vv)$ by replacing different monomial means by different variables (here $\vv$ is a vector containing the $\x^\a m_\l$).
\end{proof}

\begin{teo}\label{ThmEvenPartialSym}
\begin{align*}
\EEE\SS_{2d}^*=\{\z\in\R^{\pi(d)}:\M_{2d}'\text{ is positive semidefinite} \},
\end{align*}
where $\M_{2d}'$ is a matrix whose entries are indexed by the pairs $((\a,\l),(\a',\l'))$ with $\a,\a'\in\N^d$, $\l,\l'$ even partitions of size at most $d$, $|\a|+|\l|=|\a'|+|\l'|=d$, and its $((\a,\l),(\a',\l'))$ entry is the variable $z_{(\a+\a')\l\l'}$ if $\a+\a'$ has all even entries, or $0$ otherwise.
\end{teo}

A way of comparing $\SS_{2d}^*$ and $\PP_{2d}^*$ (or the even analogues) is to compare their extreme rays. The extreme rays of the latter are point evaluations (see Section \ref{sectionMC}) and the extreme rays of the former can be understood via the following restatement of a well-known lemma \cite[Lemma 2.2]{Ble2012}. 

Let $\L=(\L_{ij})$ be a symmetric $N\times N$ matrix whose entries are real linear forms in the variables $x_1,\dots,x_s$. 

The \emph{spectrahedral cone} defined by $\L$ is $$K_\L:=\{\vv\in\R^s\,|\,\exists A\in S_+^N: v_1x_1+\dots+v_sx_s=\langle\L,A\rangle\}.$$


\begin{lemma}\label{xraysLemma}
Let $\L$ and $K_\L$ be as above, then $\vv\in\R^s$ spans a extreme ray of $K_\L^*$ if and only if the kernel of $\L(\vv)$ is maximal with respect to inclusion: if $\ker\L(\vv)\subsetneq\ker\L(\ww)$ for some $\ww\in\R^s$ then $\L(\ww)$ is the zero matrix.
\end{lemma}

\begin{prop}\label{Prop:extraysE6}
The extreme rays of $\EEE\SS_6^*$ are precisely the ones in the direction of $(1,t,t^2)$ or $(0,0,1)$ for $t\ge1$.
\end{prop}
\begin{proof}
We use Lemma \ref{xraysLemma} to find the extreme rays of $\EEE\SS_6^*$.

By Proposition \ref{propE6} we have $\EEE\SS_6^*=\left\{(a,b,c)\in\R^3:Q(a,b,c)\succeq0\right\}$ where $$Q(x,y,z):=\begin{bmatrix}x & y & 0\\ y & z & 0\\ 0 & 0 & y-x\end{bmatrix}.$$ The kernel of $Q$ is maximal when its upper left $2\times2$ block is of rank one, so we suppose this is the case. If $x=y=0$ then $Q(0,0,1)$ has maximal kernel because it is 2-dimensional; if $x=y\ne0$ then $x=y=z$ and $Q(1,1,1)$ has again maximal kernel. If $x\ne y$ then the $Q(a,b,c)$ with $b>a>0$ and $c=\frac{b^2}a$ give maximal kernels among all the $Q$ despite being 1-dimensional. Setting $a=1$ we obtain that each $(1,b,b^2)$ with $b>1$ spans a extreme ray of $\EEE\SS_6^*$, so together with $(0,0,1)$ and $(1,1,1)$ they are all the extreme rays of $\EEE\SS_6^*$.
\end{proof}



We will use the above lemma to prove that $\EEE\SS_6=\EEE\PP_6$ and that $\EEE\SS_8=\EEE\PP_8$, but first we need to understand the limit moment cone.

\section{The Limit Moment Cone}\label{sectionMC}
Forms in $\P_{n,2d}^S$ are forms in $H_{n,2d}^S$ that give nonnegative values when evaluated at each point in $\R^n$. Hence, similar to Lemma 4.18 in \cite{BPT2012}, the dual cone $(\P_{n,2d}^S)^*\subset(H_{n,2d}^S)^*$ is the conical hull of point evaluations. By choosing the power mean basis for $H_{n,2d}^S$, point evaluations can be identified with points in the image of the map $\R^n\to\R^{\pi(2d)}$ that sends $\vv\mapsto(p_\l(\vv):\l\vdash2d)$ (where we choose the lexicographic order on partitions). This map can be factored through the \emph{moment map} $\upmu_{n,2d}:\R^n\to\R^{2d}$ (which we call $\upmu$ when clear from context) that sends $\vv\mapsto(p_1(\vv),\dots,p_{2d}(\vv))$ and through $\Phi_{2d}:\R^{2d}\to\R^{\pi(2d)}$ (which we call $\Phi$), the monomial map that sends $(p_1(\vv),\dots,p_{2d}(\vv))\mapsto(p_\l(\vv):\l\vdash 2d)$. 
\begin{center}
\begin{tikzcd}
\R^n\arrow[r, "\upmu"]& \R^{2d}\arrow[r, "\Phi"]& \R^{\pi(2d)}
\end{tikzcd}
\end{center}

Hence, $(\P_{n,2d}^\rho)^*=\cone(\Phi(\upmu(\R^n)))$ . This is a closed set since the conical hull over a compact set that does not contain the origin is closed (the image of the unit sphere $S^{n-1}$ under the continuous homogeneous map $\Phi\circ\upmu$ is compact and does not contain the origin).

To give a similar description of the limit moment cone $\PP_{2d}^*$ we need to understand what happens when $n$ goes to infinity. 

\textbf{Nonnegative symmetric forms and univariate sums of squares}. Observe that the map $\upmu_{n,2d}$ is the moment map of a uniform probability measure on the real line supported on $n$ points. Then recall that the dual cone to nonnegative univariate polynomials of degree $\le2d$ is the cone spanned by truncated moment sequences (the first $2d$ moments) of probability measures on the real line (see Chapter 3 of \cite{marshall2008positive} or, for a thorough treatment of the moment problem, \cite{schmudgen2017moment}). Now, since nonnegative univariate polynomials are sums of squares, and the dual cone to univariate sums of squares of degree $\le2d$ is the set of points $\y\in\R^{2d}$ satisfying
\begin{align}\label{momenteqtn}
    \begin{bmatrix}
    1&y_1&\dots&y_d\\
    y_1&y_2&\dots&y_{d+1}\\
    \vdots&\vdots&\ddots&\vdots\\
    y_d&y_{d+1}&\dots&y_{2d}
    \end{bmatrix}\succeq0,
\end{align}
then every such point corresponds to the sequence of moments of a probability measure
on the real line. 

Since every probability measure on the real line can be approximated by conical combinations of atomic uniform probability measures, the limit set $\HH_{2d}:=\overline{\lim_{n\to\infty}\upmu_{n,2d}(\R^n)}$ turns out to be precisely the set of points $\y\in\R^{2d}$ satisfying (\ref{momenteqtn}). 

So $\HH_{2d}$ is a \emph{Hankel spectrahedron}, which is considerably simpler than the limit set that would result from considering power sums instead of averages of power sums in the map $\upmu$ \cite{acevedo2023wonderful}. A \emph{Hankel} matrix $H$ is a square matrix whose entries on the same antidiagonal are equal, i.e., if $i+j=k+l$ then $h_{ij}=h_{kl}$. For $\y=(y_1,\dots,y_n)$ let $\H(\y)$ be the Hankel matrix $(y_{i+j-2})_{i,j}$ of size $\lceil\frac{n+1}2 \rceil$, where $y_0:=1$, and $\H'(\y)$ the Hankel matrix $(y_{i+j-1})$ of size $\lceil\frac{n}2\rceil$. Notice that $\H'$ is the submatrix of obtained by removing the first column and last row of $\H$.

\begin{prop}\label{momentmap}
The set $\HH_{2d}$ is a Hankel spectrahedron. Namely, 
\begin{align*}\label{moment}
\HH_{2d}=\{\y\in\R^{2d}:\, \H(\y)\succeq0\}.
\end{align*}
\end{prop}



As in the case of finitely many variables we will have $\PP_{2d}^*=\conv(\Phi(\HH_{2d}))$, as we will see in Corollary \ref{CorConvexptevals}.

We similarly construct $\EEE\PP_{2d}^*$ by considering the map of even power means $\upmu^{\varepsilon}_{n,2d}:=(p_2,p_4,\dots,p_{2d})$. The corresponding limit set $\HH_{2d}':=\overline{\lim_{n\to\infty}\upmu^{\varepsilon}_{n,2d}(\R^n)}$ is also a spectrahedron, now an intersection of two Hankel spectrahedrons.

The moment problem on the real line allowed us to give a description of $\HH_{2d}$ and the moment problem on $[0,\infty)$ will allow us to do the same for $\HH_{2d}'$. 

\begin{prop}\label{PropStieltjes}
\begin{align*}
\HH_{2d}'=\{\y\in\R^{d}:\, \H(\y)\succeq0,\,\, \H'(\y)\succeq0\}.
\end{align*}
\end{prop}
\begin{proof}
This follows from the fact that every univariate polynomial of degree $\le2d$ that is nonnegative on $[0,\infty)$ is of the form $f(t)+tg(t)$ where $f(t)$ and $g(t)$ are sums of squares, and the fact that the set $\upmu_{n,2d}^\varepsilon(\R^n)$ coincides with the image of the map $\R^n_{\ge0}\to\R^d$ that sends $\x$ to $(p_1(\x),\dots,p_d(\x))$. The result then follows from Stieltjes theorem \cite[Corollary 3.1.3 and Example 3.1.8]{marshall2008positive}.
\end{proof}

From the above discussion it follows that the limit moment cones have the expected description.
\begin{cor}\label{CorConvexptevals}
 \begin{align*}
 \PP_{2d}^*&=\cone(\Phi(\HH_{2d})),\\
 \EEE\PP_{2d}^*&=\cone(\Phi(\HH'_{2d})).
 \end{align*}
\end{cor}
\begin{proof}
We prove that the conical hulls are closed. We consider the case of $\HH_{2d}$ and then the even case follows analogously. Observe that $\HH_{2d}$ is a spectrahedron, so it is convex, and that in $\HH_{2d}$ there is no line segment containing the origin. Let $C_{2d}$ the intersection of $\HH_{2d}$ with the compact hypersurface $\sum_{i=1}^{2d}X_i^{2d!/i}=1$ (or $\sum_{i=1}^{2d}X_i^{4d!/2i}=1$ for the even case). So $C_{2d}$ is compact an its convex hull does not contain the origin, and therefore the same is true for $\Phi(C_{2d})$. Also $\HH_{2d}$ is  closed under the weighted scaling $(t,t^2,\dots,t^{2d})$ for $t\in\R$ (or $(t^2,t^4,\dots,t^{4d})$ for the even case) because the Hadamard product of positive semidefinite matrices is positive semidefinite.
Each point in $\HH_{2d}$ belongs to exactly one level hypersurface of $\sum_{i=1}^{2d}X_i^{2d!/i}$ and it is the image of a unique point in $C_{2d}$ under a unique weighted rescaling of the form $(t,t^2,\dots,t^{2d})$. Hence $\cone(\Phi(\HH_{2d}))=\cone(\Phi(C_{2d}))$.
\end{proof}

With the methods developed above we now can prove that $\SS_4=\PP_4$ (\cite{blekherman2020symmetric}[Theorem 2.9]), $\EEE\SS_6=\EEE\PP_6$ and $\EEE\SS_8=\EEE\PP_8$.

\begin{rem}
If all the extreme rays of $\SS_{2d}^*$ happen to be spanned by point evaluations or limits of them (i.e., by points in $\HH_{2d}$) then $\SS_{2d}^*=\PP_{2d}^*$ because $\PP_{2d}^*$ is the conical hull of point evaluations and $\PP_{2d}^*\subseteq\SS_{2d}^*$. The same holds analogously for the even symmetric case.
\end{rem}

\begin{teo}\label{thmS4}
$\SS_4=\PP_4$.
\end{teo}
\begin{proof}
With the aid of Lemma \ref{xraysLemma} we find the extreme rays of $\SS_4^*$ and then prove that they come from point evaluations via Proposition \ref{momentmap}. 

From Proposition \ref{prop4} it follows that $\SS_4^*$ is the spectrahedron $(a,b,c,d,e)$ given by $X:=\begin{bmatrix}a & b\\ b & c\end{bmatrix}\succeq0$ and $Y:=\begin{bmatrix}b-a & d-b\\ d-b & e-c\end{bmatrix}\succeq0$. 

Suppose $r:=(a,b,c,d,e)$ spans a extreme ray of $\SS_4^*$. By the \emph{kernel of} $r$ we mean the kernel of the block diagonal matrix $X\oplus Y$.

If $\rank{X}=0$ then $b-a=0$ and so $d-b=0$ (because $Y\succeq0$), hence the only extreme ray in this case is spanned by $(0,0,0,0,1)$ because its kernel $(*,*,*,0)$ is 3-dimensional, so it is maximal.

If $\rank{Y}=0$ then $a=b=d$ and $c=e$. If $a=0$ then $(0,0,1,0,1)$ spans a extreme ray with kernel $(*,0,*,*)$. If $a\ne0$ then note that $\rank{X}\ne0$ and also $\rank{X}\ne2$ because otherwise the kernel of $r$ would be strictly contained in $(*,0,*,*)$. So $\rank{X}=1$ and then $a=b=c=d=e$ and the only extreme ray in this case is spanned by $(1,1,1,1,1)$ because its kernel is 3-dimensional.

If $\rank{X}=2$ then the $\ker{r}\subsetneq(*,0,*,*)$ or if $\rank{Y}=2$ then $\ker{r}\subsetneq(*,*,*,0)$, so no extreme rays in both cases.

With Proposition \ref{momentmap} one checks that the above extreme rays come from point evaluation.

The only remaining case is $\rank{X}=\rank{Y}=1$. So we have $\det(X)=\det(Y)=0$, i.e., $ac=b^2$ and $(b-a)(e-c)=(d-b)^2$. If $a=0$ then $X\succeq0$ implies $b=0$, and $b-a=0$ implies $d-b=0$ since $Y\succeq0$, so $\ker{r}\subsetneq(*,0,*,*)$, so $a\ne0$. 
So assume $a=1$, and then $c=b^2$, $b\ge1$ and $e\ge b^2$ because $Y\succeq0$. So we have that $(b-1)(e-b^2)=(d-b)^2$ or, by expanding, $be-b^3-e-d^2+2bd=0$. The left hand side of this equality is precisely the determinant of $Z:=\begin{bmatrix}1&1&b\\ 1&b&d\\ b&d&e\end{bmatrix}$. Now observe that $Z\succeq0$ because all its principal minors are nonnegative (the only one that is not trivial to see is $be\ge d^2$ but let $b=1+p$, $e=b^2+q$ with $p,q\ge0$, so $pq=(d-b)^2$, and so $d=b\pm\sqrt{pq}$. To see $be\ge d^2$ we need $(1+p)(b^2+q)\ge(b\pm\sqrt{pq})^2$ or equivalently $q+pb^2\ge\pm2b\sqrt{pq}$ which follows directly from the AM-GM inequality). So the set of points $(a,b,c,d,e)=(1,b,b^2,d,e)$ satisfying the above conditions belong to $\Phi(\HH_4)$ because $(1,b,b^2,d,e)=\Phi(1,b,d,e)$ and $Z\succeq0$. So all the extreme rays of $\SS_4^*$ come from point evaluation.
\end{proof}

\begin{teo}\label{teoES6}
$\EEE\SS_6=\EEE\PP_6$.
\end{teo}
\begin{proof}
Observe that the extreme rays of $\EEE\SS_6^*$ (Proposition \ref{Prop:extraysE6}), $(0,0,1)$ and the $(1,t,t^2)$ with $t\ge1$, belong to $\Phi(\HH_6')$: in fact, in this special case, $\Phi$ maps $(0,0,1)$ to $(0,0,1)$ and $(1,t,t^2)$ to $(1,t,t^2)$, so  by Proposition \ref{PropStieltjes} we only need to check that $(0,0,1)$ and $(1,t,t^2)$ belong to $\HH_6'$ to verify they are point evaluations, which follows since the pairs of matrices $\begin{bmatrix}
1 & 0\\ 0 & 0
\end{bmatrix},\begin{bmatrix}
0 & 0\\ 0 & 1
\end{bmatrix}$ and $\begin{bmatrix}
1 & 1\\ 1 & t
\end{bmatrix},\begin{bmatrix}
1 & t\\ t & t^2
\end{bmatrix}$ are positive semidefinite for $t\ge1$.
\end{proof}

A little more work is required for even symmetric octics.

\begin{teo}\label{teoES8}
$\EEE\SS_8=\EEE\PP_8$.
\end{teo}
\begin{proof}
See Appendix (\ref{APPsses8=nnes8}).
\end{proof}

\section{Tropicalization}\label{sectiontrop}
We have seen in Section \ref{SectionPseudomomentCone} that the cones $\SS_{2d}^*$ and $\EEE\SS_{2d}^*$ are spectrahedral cones which can be described by explicit matrices we can compute. In Section \ref{sectionMC} we saw that the limit moment cones $\PP_{2d}^*$ and $\EEE\PP_{2d}^*$ are the conical hull of the image of Hankel spectrahedra under a monomial map. Despite these descriptions, comparing the limit cones directly seems difficult. However, by tropicalizing the limit cones we will be able to differentiate between nonnegative polynomials and sums of squares starting in degree $6$.
The tropicalization of a semialgebraic set is a polyhedral complex \cite{alessandrini2013logarithmic}, however our cones also have the property of being closed under the Hadamard (coordinate-wise) multiplication, which forces their tropicalizations to be convex polyhedral cones \cite{blekherman2020tropicalization}. Further, after tropicalization, monomial maps become linear maps and the conical hull becomes the \emph{tropical conical hull}. Thus we are left with the problem of comparing rational polyhedral cones, which we can handle it computationally. We find discrepancy in the first few cases, then use Proposition \ref{onceNEQalwaysNEQ} and Proposition \ref{PropEven} to conclude the strict inclusion for all larger degree cases.   

The tropicalization of a subset $\SSS\subset\R_{\ge0}^s$ is its logarithmic limit set \cite{alessandrini2013logarithmic} 
\begin{align*}
    \trop(\SSS)=\overline{\lim_{t\to\infty}\{(\log_t x_1,\dots,\log_t x_s):x\in\SSS\cap\R_{>0}^s\}}.
\end{align*}

A set $\SSS\subseteq\R^s$ has \emph{the Hadamard property} if whenever $\uu$ and $\vv$ belong to $\SSS$ then their Hadamard (coordinate-wise) product $\uu\circ\vv:=(u_1v_1,\dots,u_sv_s)$ also lies in $\SSS$.

Observe that $\EEE\PP_{2d}^*\subset\R^{\pi{d}}_{\ge0}$. We describe $\trop(\EEE\PP_{2d}^*)$ by first tropicalizing the image of the moment map, then taking a linear map (the \emph{tropicalization} of the monomial map $\Phi$) and finally the \emph{tropical conical hull}. Informally, this relies on the fact that semialgebraic subsets of $\R^d_{\ge0}$ can be pushed commutatively through the following diagrams:

\begin{center}
\begin{equation}\label{commdiag}
\begin{tikzcd}
\R_{\ge0}^d\arrow[r, "\Phi"]\arrow[d, "\trop"]& \R_{\ge0}^{\pi(d)}\arrow[r, "\cone"]\arrow[d, "\trop"]& \R_{\ge0}^{\pi(d)}\arrow[d,"\trop"]\\
\R^d\arrow[r, "\tilde{\Phi}"]& \R^{\pi(d)}\arrow[r, "\tcone"] &\R^{\pi(d)}
\end{tikzcd}
\end{equation}
\end{center}

In particular $$\trop(\EEE\PP_{2d}^*)=\trop(\cone(\Phi(\HH_{2d}')))=\tcone(\tilde{\Phi}(\trop(\HH_{2d}')))$$
as we will see in Proposition \ref{tropPSD}.

\begin{teo}\label{tropSOS}
$\trop(\SS_{2d}^*)$ and $\trop(\EEE\SS_{2d}^*)$ are given by the linear inequalities coming from the $2\times2$ principal minors of the matrices $\tilde{\RR}_{2d}$ and $\tilde{\RR}_{2d}^\varepsilon$ respectively.
\end{teo}
\begin{proof}
We introduce the following notation from section 4.1 of \cite{blekherman2020tropicalization}:
Let $M$ be a symmetric matrix filled with monomials in the set of variables $y_1,\dots,y_s$ for which no $2\times2$ principal minor is identically zero. For a point $\vv\in\R_{\ge0}^s$, let $M(\vv)$ denote the matrix obtained by evaluating each entry of $M$ at $\vv$. Consider the set 
\begin{align*}
S_M:=\{\vv\in\R_{\ge0}^s:M(\vv)\succeq0\}
\end{align*}
and its superset
\begin{align*}
S_M^{2\times2}:=\{\vv\in\R_{\ge0}^s:\text{ all }2\times2\text{ principal minors of }M(\vv)\text{ are nonnegative}\}.
\end{align*}
Observe both $S_M$ and $S_M^{2\times2}$ have the Hadamard property and $\1\in S_M$. 

Since $\tilde{\RR}_{2d}$ is a matrix filled with monomials then $S_{\tilde{\RR}_{2d}}$ has the Hadamard property and $\1\in S_{\tilde{\RR}_{2d}}$, so by Theorem 4.4 of \cite{blekherman2020tropicalization} we only need to prove that no $2\times2$ principal minor of $\tilde{\RR}_{2d}$ is identically zero and that $S_{\tilde{\RR}_{2d}}^{2\times2}$ has non-empty interior. By Lemma \ref{basis2} the $2\times2$ principal minors of $\tilde{\RR}_{2d}$ are of the form $\begin{vmatrix}
y_{(2\a)\l\l} & y_{(\a+\b)\l\u}\\
y_{(\a+\b)\l\u} & y_{(2\b)\u\u}
\end{vmatrix}$ where $\a,\b\in\N^d$, $\l,\u\in\Lambda$, $|\a|+|\l|=|\b|+|\u|=d$, and $\a\ne\b$ or $\l\ne\u$ since we consider different terms $\x^\a m_\l$ for computing $\RR_{2d}$. The above determinant is identically zero if and only if $(2\a)\l\l=(2\b)\u\u=(\a+\b)\l\u$ as partitions, so assume this is the case. 

If $\a=\b$ then clearly $\l=\u$, so suppose $\a\ne\b$. Note that if $\a_i=\b_i$ then $2\a_i=2\b_i=\a_i+\b_i$ so if we remove $2\a_i$ from $2\a$, $2\b_i$ from $2\b$ and $\a_i+\b_i$ from $\a+\b$ then again the above equalities hold, so we can assume $\a_i\ne\b_i$ for each $i\in\supp(\a+\b)$. The minimum part of $\a+\b$ is strictly larger than the minimum part of $2\a$ or $2\b$ since we either have the inequality $2\a_i<\a_i+\b_i<2\b_i$ or its reverse for each $i\in\supp(\a+\b)$. This implies that the minimum part $l$ of $(\a+\b)\l\u$ is either the minimum part of $\l$ or $\u$, so suppose it is $\l$ and that $l$ appears no less times in $\l$ than in $\u$. Since $(2\a)\l\l=(\a+\b)\l\u$ and $l$ does not appear in $\a+\b$ then it must appear in $\u$ the same number of times it appears on $\l$. So we can remove all appeareances of $l$ in $\l$ and $\u$ and get the same equalities. We can repeat this only a finite number of times so $\l=\u$. Hence $2\a=2\b=\a+\b$ as partitions, but the minimum part of $\a+\b$ is stricly larger than the minimum part of $2\a$ or $2\b$, a contradiction. 

As a consequence, since $\SS_{2d}^*$ is full-dimensional and $S_{\tilde{\RR}_{2d}}$ is a subset of $S_{\tilde{\RR}_{2d}}^{2\times2}$, Theorem 4.4 of \cite{blekherman2020tropicalization} implies that $\trop(\SS_{2d}^*)=\trop(S_{\tilde{\RR}_{2d}}^{2\times2})$. The proof is the same for $\tilde{\RR}^\varepsilon_{2d}$.
\end{proof}


The \emph{tropical conical hull} of a set $\SSS\subset\R^s$ is the set of all tropical linear combinations of points in $\SSS$,
\begin{align*}
    \tcone(\SSS)=\{(a_1\odot\s_1)\oplus(a_2\odot\s_1)\oplus\cdots\oplus(a_r\odot\s_r):r\in\N, a_1,\dots,a_r\in\R, \s_1,\dots,\s_r\in\SSS\}.
\end{align*} 
A set is called \emph{tropically convex} if it equals its tropical conical hull. 

\begin{prop}[Lemma 2.4\cite{blekherman2022moments}]\label{Tcone}
For $\SSS\subset\R^s$,  
\begin{align*}
\tcone(\SSS)=\bigcap_{i=1}^{s}(\SSS+V_i).
\end{align*}
where $V_i:=\{-x\in\R^s:x_1\oplus x_2\oplus\cdots\oplus x_s=x_i\}$.
\end{prop}




For a monomial map $\varphi:\R^n\to\R^m$ given by $\varphi(\x)=(\x^{\aa_1},\dots,\x^{\aa_m})$, we define its \emph{tropicalization} $\tilde{\varphi}$ to be the linear map $\tilde{\varphi}:\R^n\to\R^m$ given by $\tilde{\varphi}(\x)=(\aa_1\cdot\x,\dots,\aa_m\cdot\x)$. Notice that for any point $\x\in\R^n_{>0}$ we have $\log(\varphi(\x))=\tilde{\varphi}(\log(\x))$, hence $\trop\Phi(\HH)=\tilde{\Phi}(\trop\HH)$ for any $\HH\subset\R^n_{\ge0}$. The map $\tilde{\Phi}$ is the \emph{tropicalization} of $\Phi$ and this guarantees the commutativity of the left part of diagram \ref{commdiag}. The commutativity of the right part follows from the next lemma.

\begin{lemma}[Proposition 2.3 \cite{blekherman2022moments}]\label{LemmaConvTropTconv}
If $\K\subset\R^n_{\ge0}$ is semialgebraic then
\begin{align*}
\trop(\cone\K)=\tcone(\trop\K).
\end{align*}
\end{lemma}

\begin{cor}\label{tropPSD}
$$\trop(\EEE\PP_{2d}^*)=\trop(\cone(\Phi(\HH_{2d}')))=\tcone(\tilde{\Phi}(\trop\HH_{2d}')).$$
\end{cor}
\begin{proof}
By Corollary \ref{CorConvexptevals} $\EEE\PP_{2d}^*=\cone(\Phi(\HH_{2d}'))$. By Proposition \ref{PropStieltjes} $\HH_{2d}'$ is a semialgebraic set, hence $\Phi(\HH_{2d}')$ is semialgebraic, and so by Lemma \ref{LemmaConvTropTconv} we have $$\trop(\cone(\Phi(\HH_{2d}')))=\tcone(\trop\Phi(\HH_{2d}'))$$ which equals, by the commutativity of the left part of diagram \ref{commdiag}, $\tcone(\tilde{\Phi}(\trop\HH_{2d}'))$. 
\end{proof}

\begin{prop}\label{PropTropEvenHankel}
$\trop(\HH_{2d}')$ is the convex polyhedral cone defined by the $d-1$ inequalities
\begin{align*}
    -2z_1+z_2&\ge0\\
    z_1-2z_2+z_3&\ge0\\
    z_2-2z_3+z_4&\ge0\\
    &\vdots\\
    z_{d-2}-2z_{d-1}+z_d&\ge0.
\end{align*}
\end{prop}
\begin{proof}
We use the notation of the proof of Theorem \ref{tropSOS}. Observe that $\trop(\HH_{2d}')=\trop(S_{\H\oplus\H'}^{2\times2})=\trop(S_\H^{2\times2})\cap\trop(S_{\H'}^{2\times2})$ where $\H=\H(y_1,\dots,y_d)$ and $\H'=\H'(y_1,\dots,y_d)$. By Theorem 4.4 of \cite{blekherman2020tropicalization} it is enough to prove that no $2\times2$ principal minor of $\H\oplus\H'$ vanishes, which is clear, and that $S_{\H\oplus\H'}^{2\times2}$ has an interior in $\R^d$. The point $(\frac12,\frac13,\dots,\frac1{d+1})$ is an interior point of $S_{\H\oplus\H'}$ since $\H(\frac12,\frac13,\dots,\frac1{d+1})$ is a Hilbert matrix and Hilbert matrices are \emph{totally positive} (every submatrix has positive determinant) so its submatrix $\H'(\frac12,\frac13,\dots,\frac1{d+1})$ is also positive definite. Finally, the inequalities coming from $2\times2$ principal minors with consecutive rows conically span the inequalities coming from the other $2\times2$ principal minors.
\end{proof}


\section{Proofs of main theorems}\label{sectionApp}


We use the following algorithm, justified by Proposition \ref{Tcone}, to compute tropical convex hulls in SAGE. In our case tropical conical hulls coincide with tropical convex hulls since the map $\Phi\circ\mu$ is homogeneous.

\begin{alg}\label{Algorithm}
Given a polyhedron $P$ return its tropical convex hull.
\begin{verbatim}
    def tropConv(P):
    d = P.ambient_dim();
    I = matrix.identity(d);
    E = I.rows();
    sectors = [Polyhedron(rays=[E[j] for j in range(d) if j != i], 
               lines=[Sequence([1 for k in range(d)])]) for i in range(d)];
    MinkSums = [P.minkowski_sum(S) for S in sectors];
    tconv = MinkSums[0];
    for i in range(1,d):
        tconv = tconv.intersection(MinkSums[i])
    return tconv
\end{verbatim}
\end{alg}

We use the above algorithm to compute $\trop(\EEE\PP_{2d}^*)$ for $d=5,6,7,8$ (and $\trop((\PP_6')^*)$ defined below). The input is $P=\tilde{\Phi}(\trop(\HH_{2d}'))$. Having the defining inequalities for $\trop(\HH_{2d}')$ (Proposition \ref{PropTropEvenHankel}) we compute its extreme rays (again using SAGE), and so the image of these under the linear map $\tilde{\Phi}$ will conically span $\tilde{\Phi}(\trop(\HH_{2d}'))$.

\begin{proof}[Proof of Theorem \ref{Thm1}]
We have $\SS_2=\PP_2$ and $\SS_4=\PP_4$ (Theorem \ref{thmS4}). By Proposition \ref{onceNEQalwaysNEQ} it is enough to prove that that $\SS_6\subsetneq\PP_6$. 

We call a partition \emph{odd} if it has at least one odd part appearing an odd number of times. Consider the linear subspace of limit sextics without terms of the form $\m_\l$ where $\l$ is odd. Namely, 
\begin{align}
H=\left\{\sum_{\l\vdash6}c_\l\m_\l:c_{3,1^3}=c_{3,2,1}=c_{5,1}=0\right\}.
\end{align}
Let $\SS_6':=\SS_6|_H$ and $\PP_6':=\PP_6|_H$ be the slices of the limit cones with the linear subspace $H$. 
The dual cone $(\PP_6')^*$ is the limit set of the conical hull of $\Phi_6'(\HH_6)$ where
\begin{align*}
\Phi_6':(p_1,p_2,p_3,p_4,p_5,p_6)\mapsto(p_1^6,p_2p_1^4,
p_2^2p_1^2,p_2^3,p_3^2,p_4p_1^2,p_4p_2,p_6)
\end{align*}
is the monomial map given by the partitions of 6 where each odd part appears an even number of times.
Hence $(\PP_6')^*\subseteq\R_{\ge0}^8$. Now compute $\trop((\PP_6')^*)$ with Algorithm \ref{Algorithm} to obtain that $$z_{1^6}-3z_{4,1^2}+z_{4,2}+z_{6}\ge0$$ defines one of its facets. 

Finally observe that the above inequality does not define a facet of $\trop((\SS_6')^*)$ since its facets are given by $2\times2$ principal minors of $\tilde{\RR}_6$ that do not contain $\m_\l$'s where $\l$ is odd. They are of the form $z_\u-z_\nu$ or $z_\u-2z_\omega+z_\nu$ where $\u,\nu,\omega$ are different non-odd partitions of 6. 

 Hence $\trop((\PP_6')^*)\subsetneq\trop((\SS_6')^*)$, which implies $\SS_6'\subsetneq\PP_6'$, and therefore $\SS_6\subsetneq\PP_6$.
\end{proof}

%

The proof of Theorem \ref{Thm2Even} relies on computing $\trop(\EEE\PP_{2d}^*)$ for $d=5,6,7,8$ with Algorithm \ref{Algorithm} and observing that we also obtain facets that are not of the form $z_\mu-z_u$ or $z_\mu-2z_\omega+z_\nu$.

\begin{proof}[Proof of Theorem \ref{Thm2Even}]
Trivially $\EEE\SS_2=\EEE\PP_2$. Given that $\SS_4=\PP_4$ (Theorem \ref{thmS4}) then $\EEE\SS_4=\EEE\PP_4$, also $\EEE\SS_6=\EEE\PP_6$ and $\EEE\SS_8=\EEE\PP_8$ (Theorem \ref{teoES6} and Theorem \ref{teoES8}). By Proposition \ref{Prop5678} we have $\EEE\SS_{2d}\subsetneq\EEE\PP_{2d}$ for $d=5,6,7,8$, and Proposition \ref{PropEven} implies the strict inclusion for all larger $d$.
\end{proof}


\subsection{Explicit example}



The facets of $\trop(\EEE\PP_{2d}^*)$ that are not facets of $\trop(\EEE\SS_{2d}^*)$ allow us to find explicit examples of nonnegative symmetric forms that are not sums of squares for any sufficiently large number of variables. 

The limit forms $a\m_6+b\m_{4,2}+c\m_{1^6}-3\m_{4,1^2}$ with $abc=1$ and $a,b,c\ge0$ are nonnegative because for any fixed number of variables $n\ge6$ we have,
$$\frac{ap_6+bp_4p_2+cp_1^6}3\ge\sqrt[3]{p_6p_4p_2p_1^6}\ge p_4p_1^2$$
which follows from the AM-GM inequality and majorization inequalities for power means \cite{cuttler2011inequalities}.

After trying different values for $a,b,c$ we found that $\frac12\m_6+\frac12\m_{4,2}+4\m_{1^6}-3\m_{4,1^2}$ is not a sum of squares in the limit. This can be verified by setting a semidefinite program constructed from the matrices $\Q_{\emptyset,6}$, $\Q_{(1),6}$ and $\Q_{(1^2),6}$ in Proposition \ref{prop6}, and obtaining a rational infeasibility certificate using YALMIP \cite{Lofberg2004} toolbox on MATLAB.

\bibliographystyle{alpha}
\bibliography{ssos}

\newpage
\section{Appendix}

\begin{ex}[symmetry basis for $H_{n,3}$]\label{apphn3}
For all $n\ge6$
\begin{align*}
H_{n,3}\cong_{\mathcal{S}_n}(\Sc_{(n)})^{\oplus3}\oplus(\Sc_{(n-1,1)})^{\oplus4}\oplus\Sc_{(n-2,1,1)}\oplus(\Sc_{(n-2,2)})^{\oplus2}\oplus\Sc_{(n-3,3)}.
\end{align*}
A symmetry basis for $(H_{n,3})_{(n)}$ is simply $\{m_{1^3}^{(n)},m_{2,1}^{(n)},m_3^{(n)}\}$.
We have $h_{(1),3}=4$ so the following forms constitute a symmetry basis for $(H_{n,3})_{(n-1,1)}$ since they are linearly independent
\begin{eqnarray*}
Y_{(n-1,1)}(x_1x_3x_4)&=&\aa_{(n-1,1)}\frac12(x_1x_3x_4-x_2x_3x_4)=x_1m_{1^2}(\x_{[1]})-m_{1^3}(\x_{[1]})\\
Y_{(n-1,1)}(x_1^2x_3)&=&\aa_{(n-1,1)}\frac12(x_1^2x_3-x_2^2x_3)=x_1^2m_1(\x_{[1]})-m_{2,1}(\x_{[1]})\\
Y_{(n-1,1)}(x_1^2x_2)&=&\aa_{(n-1,1)}\frac12(x_1^2x_2-x_1x_2^2)=x_1^2m_1(\x_{[1]})-x_1m_2(\x_{[1]}),\\
Y_{(n-1,1)}(x_1^3)&=&\aa_{(n-1,1)}\frac12(x_1^3-x_2^3)=x_1^3-m_3(\x_{[1]}).
\end{eqnarray*}
Alternatively one can show that applying $Y_{(n-1,1)}$ to other monomials results in the same or the negatives of the forms computed above, but multiples generate the same copy of $\Sc_{\l}$, so the above forms should be linearly independent given that $h_{(1),3}=4$. Repeating the same procedure for the rest of the isotypic components we find the following symmetry bases for $(H_{n,3})_{(n)}$, $(H_{n,3})_{(n-1,1)}$, $(H_{n,3})_{(n-2,1^2)}$, $(H_{n,3})_{(n-2,2)}$ and $(H_{n,3})_{(n-3,3)}$:
\begin{align*}
\B_{\emptyset,3}&:=\{m_{1^3},\, m_{2,1},\, m_3\},\\
\B_{(1),3}&:=\{x_1m_{1^2}-m_{1^3},\, x_1^2m_1-m_{2,1},\, x_1^2m_1-x_1m_2,\, x_1^3-m_3\},\\
\B_{(1^2),3}&:=\{x_1^2x_2-x_1x_2^2-x_2m_2-x_1^2m_1+x_2^2m_1+x_1m_2\},\\
\B_{(2),3}&:=\{x_1x_2m_1-(x_1+x_2)m_{1^2}+m_{1^3},x_1^2x_2+x_1x_2^2-(x_1^2+x_2^2)m_1-(x_1+x_2)m_2+2m_{2,1}\},\\
\B_{(3),3}&:=\{x_1x_2x_3-(x_1x_2+x_1x_3+x_2x_3)m_1+(x_1+x_2+x_3)m_{1^2}-m_{1^3}\}.
\end{align*}{\flushright$\triangle$}
\end{ex}

\begin{ex}[isotypic matrices for sextics at infinity]\label{appiso6}
We compute isotypic matrices $\Q_{\u,6}$ for each $\u\in \F_3$.
Using Lemma \ref{basis} on the bases obtained in Example \ref{apphn3} we obtain
\begin{align*}
\sym_\infty(m_{1^3}\cdot m_{1^3})=\m_{1^6},\,
\sym_\infty(m_{1^3}\cdot m_{2,1})=\m_{2,1^4},\,\sym_\infty(m_{1^3}\cdot m_3)&=\m_{3,1^3},\\
\sym_\infty(m_{2,1}\cdot m_{2,1})=\m_{2^2,1^2},\,
\sym_\infty(m_{2,1}\cdot m_3)&=\m_{3,2,1},\\
\sym_\infty(m_3\cdot m_3)&=\m_{3^2},
\end{align*}
therefore
\begin{align*}
\Q_{\emptyset,6}=\begin{bmatrix}
\m_{1^6} & \m_{2,1^4} & \m_{3,1^3}\\ \m_{2,1^4} & \m_{2^2,1^2} & \m_{3,2,1}\\ \m_{3,1^3} & \m_{3,2,1} & \m_{3^2}
\end{bmatrix}.
\end{align*}
From the symmetry basis $\B_{(1),3}$ we get for example $$(x_1m_{1^2}-m_{1^3})^2=x_1^2m_{1^2}^2-2x_1m_{1^2}m_{1^3}+m_{1^3}^2,\text{ therefore}$$
\begin{align*}
&\phantom{=}\,\,\sym_\infty(x_1^2m_{1^2}^2-2x_1m_{1^2}m_{1^3}+m_{1^3}^2)\\
&=\sym_\infty(x_1^2m_{1^2}^2)-2\sym_\infty(x_1m_{1^2}m_{1^3})+\sym_\infty(m_{1^3}^2)\\
&=\m_{2,1^4}-2\m_{1^6}+\m_{1^6}=\m_{2,1^4}-\m_{1^6}.
\end{align*}
Similarly, for the rest of the combinations we get
\begin{align*}
M=\begin{bmatrix}
\m_{2,1^4}-\m_{1^6} & \m_{3,1^3}-\m_{2,1^4} & \m_{3,1^3}-\m_{2^2,1^2} & \m_{4,1^2}-\m_{3,1^3}\\
\m_{3,1^3}-\m_{2,1^4} & \m_{4,1^2}-\m_{2^2,1^2} & \m_{4,1^2}-\m_{3,2,1} & \m_{5,1}-\m_{3,2,1}\\
\m_{3,1^3}-\m_{2^2,1^2} & \m_{4,1^2}-\m_{3,2,1} & \m_{4,1^2}-2\m_{3,2,1}+\m_{2^3} & \m_{5,1}-\m_{4,2}\\
\m_{4,1^2}-\m_{3,1^3} & \m_{5,1}-\m_{3,2,1} & \m_{5,1}-\m_{4,2} & \m_6-\m_{3^2}
\end{bmatrix}
\end{align*}
Which we may call $\Q_{(1),6}$ instead, but one can give a simpler form by considering a congruent matrix and it will not change the description of $\SS_6$ given in Theorem \ref{Qrep} since the PSD cone is invariant under congruence. Hence one can set $\Q_{(1),6}$ to be $PMP^\top$ for some nonsingular $P$.

Setting $P=\begin{bmatrix}-1&0&0&0\\ 0&-1&0&0\\ 0&-1&1&0\\ 0&0&0&-1\end{bmatrix}$ we obtain the better looking
\begin{align*}
\Q_{(1),6}=\begin{bmatrix}
\m_{2,1^4}-\m_{1^6} & \m_{3,1^3}-\m_{2,1^4} & \m_{2^2,1^2}-\m_{2,1^4} & \m_{4,1^2}-\m_{3,1^3}\\
\m_{3,1^3}-\m_{2,1^4} & \m_{4,1^2}-\m_{2^2,1^2} & \m_{3,2,1}-\m_{2^2,1^2} & \m_{5,1}-\m_{3,2,1}\\
\m_{2^2,1^2}-\m_{2,1^4} & \m_{3,2,1}-\m_{2^2,1^2} & \m_{2^3}-\m_{2^2,1^2} & \m_{4,2}-\m_{3,2,1} \\
\m_{4,1^2}-\m_{3,1^3} & \m_{5,1}-\m_{3,2,1} & \m_{4,2}-\m_{3,2,1} & \m_6-\m_{3^2}
\end{bmatrix}
\end{align*}

Applying the same procedure to the rest of the isotypic components we obtain
\begin{eqnarray*}
\Q_{(1^2),6}&=& \m_{4,2}-\m_{4,1^2}-\m_{3^2}+2\m_{3,2,1}-\m_{2^3},\\[16pt]
\Q_{(2),6}&=&\begin{bmatrix}
\m_{2^2,1^2}-2\m_{2,1^4}+\m_{1^6} & 2(\m_{3,2,1}-\m_{3,1^3}-\m_{2^2,1^2}+\m_{2,1^4})\\
2(\m_{3,2,1}-\m_{3,1^3}-\m_{2^2,1^2}+\m_{2,1^4}) & 2(\m_{4,2}-\m_{4,1^2}+\m_{3^2}-2\m_{3,2,1}-\m_{2^3}+2\m_{2^2,1^2})
\end{bmatrix},\\[16pt]
\Q_{(3),6}&=& \m_{2^3}-3\m_{2^2,1^2}+3\m_{2,1^4}-\m_{1^6}.
\end{eqnarray*}
Now notice that for all $\alpha,\beta\in\mathbb{R}$ we have $$\begin{bmatrix}\alpha\\ \beta\end{bmatrix}^\top\Q_{(2),6}\begin{bmatrix}\alpha\\ \beta\end{bmatrix}=\begin{bmatrix}\alpha\\ 2\beta-\alpha\\ -2\beta\end{bmatrix}^\top\Q_{\emptyset,6}\begin{bmatrix}
\alpha\\ 2\beta-\alpha\\ -2\beta\end{bmatrix}+2\beta^2\Q_{(1^2),6}$$ and $\Q_{(3),6}=[1,0,-1,0]\Q_{(1),6}[1,0,-1,0]^\top$. Hence $\Q_{\emptyset,6}$, $\Q_{(1),6}$ and $\Q_{(1^2),6}$ are enough for describing $\SS_6$, i.e.,
\begin{align}
\SS_6=\{\langle A_\emptyset,\Q_{\emptyset,6}\rangle+\langle A_{(1)},\Q_{(1),6}\rangle+\langle A_{(1^2)},\Q_{(1^2),6}\rangle:A_\emptyset\in S_3^+,A_{(1)}\in S_4^+,A_{(1^2)}\ge0\}.
\end{align}{\flushright$\triangle$}
\end{ex}

\begin{ex}[Even symmetric sos sextics]\label{appes6}
We can describe $\SS_6^e$ in three ways: applying $\e$ to the matrices found on Example \ref{appiso6}; applying $\e$ to the symmetry bases computed in Example \ref{apphn3} and then computing the matrices $\EE_{\u,6}$; or with partial symmetry reduction by collecting the terms of our symmetry bases which do not contain an $m_\l$ with an odd part (i.e. where $\l$ has an odd part). In the first case, by Example \ref{appiso6}, it is enough to consider $\Q_{\emptyset,6}$, $\Q_{(1),6}$ and $\Q_{(1^2),6}$, so we obtain $\e(\Q_{\emptyset,6})=0$, $\e(\Q_{(1),6})=\begin{bmatrix}\m_{2^3} & \m_{4,2}\\ \m_{4,2} & \m_6\end{bmatrix}$ and $\e(\Q_{(1^2),6})=\m_{4,2}-\m_{2^3}$ where we have abused notation and removed rows and columns of zeros. 

Hence $\SS_6^e=\left\{\left\langle A,\begin{bmatrix}\m_{2^3} & \m_{4,2}\\ \m_{4,2} & \m_6\end{bmatrix}\right\rangle+\a(\m_{4,2}-\m_{2^3}):A\in S_2^+,\a\ge0\right\}$. 

In the third case, looking at the terms found in the symmetry basis for $H_{n,3}$ in Example \ref{apphn3} we see that the only terms with no $m_\l$'s with odd parts are $x_1m_2, x_1x_2^2, x_1^3, x_2m_2, x_1^2x_2, x_1x_2x_3$, and from these we can ignore $x_1x_2x_3$ because when we multiply it with the others we will always obtain variables with odd exponents and $\sym_\infty(x_1x_2x_3)^2=\m_{2^3}$ which can also be obtained as $\sym_\infty(x_1m_2)^2$, so it gives no new information. Choosing $\w:=(x_1m_2, x_1x_2^2, x_1^3, x_2m_2, x_1^2x_2)^\top$ (which are also the terms appearing in the symmetry bases $\B_{\emptyset,3},\B_{(1),3},\B_{(1^2),3}$) we obtain the block diagonal matrix
\begin{align*}
\sym_\infty(\w\w^\top)=\begin{bmatrix}
\m_{2^3}&\m_{2^3}&\m_{4,2}\\
\m_{2^3}&\m_{4,2}&\m_{4,2}\\
\m_{4,2}&\m_{4,2}&\m_6
\end{bmatrix}\oplus\begin{bmatrix}\m_{2^3}&\m_{2^3}\\ \m_{2^3}&\m_{4,2}\end{bmatrix}
\end{align*}
where we further can ignore the second block because it equals the top $2\times2$ principal submatrix of the first block.
 
Hence we also have $\SS_6^e=\left\{\left\langle A,\begin{bmatrix}
\m_{2^3}&\m_{2^3}&\m_{4,2}\\
\m_{2^3}&\m_{4,2}&\m_{4,2}\\
\m_{4,2}&\m_{4,2}&\m_6
\end{bmatrix}\right\rangle:A\in S_3^+\right\}$.$\triangle$
\end{ex}

\begin{ex}[even symmetric sos octics]\label{appes8}
For all $n\ge 8$
\begin{align*}
H_{n,4}\cong_{\mathcal{S}_n}(\Sc_{(n)})^{\oplus5}\oplus
(\Sc_{(n-1,1)})^{\oplus7}\oplus
(\Sc_{(n-2,1^2)})^{\oplus2}\oplus
(\Sc_{(n-2,2)})^{\oplus5}\oplus
\Sc_{(n-3,2,1)}\oplus
(\Sc_{(n-3,3)})^{\oplus2}\oplus
\Sc_{(n-4,4)}.
\end{align*}
We compute the symmetry bases $\B_{\u,4}$
\begin{align*}
\B_{\emptyset,4}&=\{m_{1^4},\,\, m_{2,1^2},\,\, m_{2^2},\,\, m_{3,1},\,\, m_4\},\\
\B_{(1),4}&=\{x_1m_{1^3}-m_{1^4},\,\, x_1^2m_{1^2}-m_{2,1^2},\,\, x_1^2m_{1^2}-x_1m_{2,1},\,\,x_1^2m_2-m_{2^2},\\
&\qquad x_1^3m_1-m_{3,1},\,\, x_1^3m_1-x_1m_3,\,\, x_1^4-m_4\},\\
\B_{(1^2),4}&=\{x_1^2x_2m_1-x_1x_2^2m_1-x_2m_{2,1}-x_1^2m_{1^2}+x_2^2m_{1^2}+x_1m_{2,1},\\
&\qquad x_1^3x_2-x_1x_2^3-x_2m_3-x_1^3m_1+x_2^3m_1+x_1m_3\},\\
\B_{(2),4}&=\{x_1x_2m_{1^2} -x_1m_{1^3} -x_2m_{1^3} +m_{1^4},\\
&\qquad x_1x_2m_2 -x_1m_{2,1} -x_2m_{2,1} +m_{2,1^2},\\
&\qquad x_1^2x_2m_1 +x_1x_2^2m_1 -x_1^2m_{1^2} -x_2^2m_{1^2} -x_1m_{2,1} -x_2m_{2,1} +2m_{2,1^2},\\
&\qquad x_1^2x_2^2 -x_1^2m_2 -x_2^2m_2 +m_{2^2},\\
&\qquad x_1^3x_2 +x_1x_2^3 -x_1^3m_1 -x_2^3m_1 -x_1m_3 -x_2m_3 +2m_{3,1}\},\\
\B_{(2,1),4}&=\{2x_1^2x_2x_3-2x_1^2x_2m_1-2x_1^2x_3m_1-x_1x_2^2x_3-x_1x_2x_3^2+x_1x_2^2m_1\\
&+x_1x_3^2m_1-2x_2x_3m_2+x_2m_{2,1}+x_3m_{2,1}+2x_1^2m_{1^2} +x_2^2x_3m_1 +x_2x_3^2m_1\\
&-x_2^2m_{1^2} -x_3^2m_{1^2} +x_1x_2m_2 +x_1x_3m_2 -2x_1m_{2,1}\},\\
\B_{(3),4}&=\{x_1x_2x_3m_1-x_1x_2m_{1^2}-x_1x_3m_{1^2}-x_2x_3m_{1^2}+x_1m_{1^3}+x_2m_{1^3} +x_3m_{1^3}-m_{1^4},\\
&\qquad x_1^2x_2x_3 +x_1x_2^2x_3 +x_1x_2x_3^2 -x_1^2x_2m_1 -x_1^2x_3m_1 -x_1x_2^2m_1 -x_2^2x_3m_1\\
    &-x_1x_3^2m_1 -x_2x_3^2m_1 +x_1^2m_{1^2} +x_2^2m_{1^2} +x_3^2m_{1^2} -x_1x_2m_2 -x_1x_3m_2-x_2x_3m_2\\ &+2x_1m_{2,1} +2x_2m_{2,1}+2x_3m_{2,1} -3m_{2,1^2}\},\\
\B_{(4),4}&=\{x_1x_2x_3x_4-(x_1x_2x_3+x_1x_2x_4+x_1x_3x_4+x_2x_3x_4)m_1\\
&+(x_1x_2+x_1x_3+x_1x_4+x_2x_3+x_2x_4+x_3x_4)m_{1^2}-(x_1+x_2+x_3+x_4)m_{1^3}+m_{1^4}\}
\end{align*}

and then consider $\e(\B_{\u,4})$ to compute the $\EE_{\u,8}$ and obtain
\begin{align*}
\EE_{\emptyset,8}&=\begin{bmatrix}\m_{2^4}&\m_{4,2^2}\\ 
\m_{4,2^2}&\m_{4^2}\end{bmatrix}\\
\EE_{(1),8}&=\begin{bmatrix}\m_{4,2^2}-\m_{2^4}&\m_{6,2}-\m_{4,2^2}\\ 
\m_{6,2}-\m_{4,2^2}&\m_8-\m_{4^2}\end{bmatrix}\\
\EE_{(1^2),8}&=\m_{6,2}-\m_{4^2}\\
\EE_{(2),8}&=\begin{bmatrix}\m_{2^4}&0&2\m_{4,2^2}\\ 0&\m_{4^2}-2\m_{4,2^2}+\m_{2^4}&0\\ 
2\m_{4,2^2}&0&2\m_{6,2}+2\m_{4^2}\end{bmatrix}\\
\EE_{(2,1),8}&=\m_{4,2^2}-\m_{2^4}\\
\EE_{(3),8}&=\m_{4,2^2}-\m_{2^4}\\
\EE_{(4),8}&=\m_{2^4}
\end{align*}
Again some of the $\EE_{\u,8}$ are redundant for describing $\EEE\SS_8$. For instance, for all $\a,\b,\g\in\R$ we have $(\a,\b,\g)\EE_{(2),8}(\a,\b,\g)^\top=2\g^2\EE_{(1^2),8}+\left\langle\begin{bmatrix}\a^2+\b^2&2\a\g-\b^2\\ 2\a\g-\b^2&\b^2+4\g^2\end{bmatrix},\EE_{\emptyset,8}\right\rangle$. Also $\EE_{(2,1),8}$, $\EE_{(3),8}$ and $\EE_{(4),8}$ can be easily obtained in a similar way from $\EE_{\emptyset,8}$, $\EE_{(1),8}$ and $\EE_{(1^2),8}$. Therefore 
\begin{align*}
\EEE\SS_8=\{\langle A_\emptyset,\EE_{\emptyset,8}\rangle+\langle A_{(1)},\EE_{(1),8}\rangle+\alpha\EE_{(1^2),8}:A_\emptyset\in S_2^+,A_{(1)}\in S_2^+,\alpha\ge0\}.
\end{align*} 
\end{ex}

\begin{ex}\label{appes10} 
For all $n\ge10$
\begin{align*}
H_{n,5}\cong_{\mathcal{S}_n}\Sc_{(n)}^{\oplus7}\oplus
\Sc_{(n-1,1)}^{\oplus12}\oplus
\Sc_{(n-2,1,1)}^{\oplus5}\oplus
\Sc_{(n-2,2)}^{\oplus9}\oplus
\Sc_{(n-3,2,1)}^{\oplus3}\oplus
\Sc_{(n-3,3)}^{\oplus5}\oplus
\Sc_{(n-4,3,1)}\oplus
\Sc_{(n-4,4)}^{\oplus2}\oplus
\Sc_{(n-5,5)}.
\end{align*}
After computing symmetry bases for each $\mathcal{S}_n$-isotypic component of $H_{n,5}$ we compute the corresponding even isotypic matrices:
\begin{align*}
\EE_{(1),10}&=\begin{bmatrix}
\m_{2^5}&\m_{4,2^3}&\m_{4,2^3}&\m_{6,2^2}\\
\m_{4,2^3}&\m_{4^2,2}&\m_{4^2,2}&\m_{6,4}\\
\m_{4,2^3}&\m_{4^2,2}&\m_{6,2^2}&\m_{8,2}\\
\m_{6,2^2}&\m_{6,4}&\m_{8,2}&\m_{10}
\end{bmatrix},\\
\EE_{(1^2),10}&=\begin{bmatrix}
\m_{4,2^3}-\m_{2^5}&\m_{4^2,2}-\m_{4,2^3}&\m_{6,2^2}-\m_{4,2^3}\\
\m_{4^2,2}-\m_{4,2^3}&\m_{6,4}-\m_{6,2^2}&\m_{6,4}-\m_{4^2,2}\\
\m_{6,2^2}-\m_{4,2^3}&\m_{6,4}-\m_{4^2,2}&\m_{8,2}-\m_{4^2,2}
\end{bmatrix},\\
\EE_{(2),10}&=\begin{bmatrix}
\m_{4,2^3}-\m_{2^5}&\m_{4^2,2}-\m_{4,2^3}&\m_{6,2^2}-\m_{4,2^3}\\
\m_{4^2,2}-\m_{4,2^3}&\m_{6,4}-\m_{6,2^2}&\m_{6,4}-\m_{4^2,2}\\
\m_{6,2^2}-\m_{4,2^3}&\m_{6,4}-\m_{4^2,2}&\m_{8,2}-\m_{4^2,2}
\end{bmatrix},\\
\EE_{(2,1),10}&=\begin{bmatrix}\m_{4^2,2}-2\m_{4,2^3}+\m_{2^5}&0\\0&\m_{6,2^2}-\m_{4^2,2}\end{bmatrix},\\
\EE_{(3),10}&=\begin{bmatrix}
\frac13\m_{2^5}&0&\m_{4,2^3}\\
0&\m_{4^2,2}-2\m_{4,2^3}+\m_{2^5}&0\\
\m_{4,2^3}&0&\m_{6,2^2}+2\m_{4^2,2}
\end{bmatrix}\\
\EE_{(3,1),10}&=\m_{4,2^3}-\m_{2^5},\\
\EE_{(4),10}&=\m_{4,2^3}-\m_{2^5},\\
\EE_{(5),10}&=\m_{2^5}
\end{align*}
It is not hard to see that only $\EE_{(1),10}$ and $\EE_{(1^2),10}$ are enough for describing $\EEE\SS_{10}$. Hence we can collect the non-odd terms from $\B_{(1),5}$ and $\B_{(1^2),5}$. We separate these terms in two sets: the ones having an odd power of $x_1$ and the ones having an even power of $x_1$: $\{x_1m_{2^2},x_1x_2^2m_2,x_1m_4,x_1^3m_2,x_1^3x_2^2,
x_1x_2^4,x_1^5\}$ and $\{x_2m_{2^2},x_1^2x_2m_2,x_2m_4,x_2^3m_2,x_1^2x_2^3,x_1^4\}$.  Multiplying an element from the first set with an element of the second will give an odd power in $x_1$ hence it will even-symmetrize to zero. From the first set we obtain the following matrix after partial symmetrization:
\begin{align}\label{partialmat10}
\RR_{10}^e=\begin{bmatrix}
\m_{2^5} & \m_{2^5} & \m_{4,2^3} & \m_{4,2^3} & \m_{4,2^3} & \m_{4,2^3} & \m_{6,2^2}\\
\m_{2^5} & \m_{4,2^3} & \m_{4,2^3} & \m_{4,2^3} & \m_{4^2,2} & \m_{6,2^2} & \m_{6,2^2}\\
\m_{4,2^3} & \m_{4,2^3} & \m_{4^2,2} & \m_{4^2,2} & \m_{4^2,2} & \m_{4^2,2} & \m_{6,4}\\
\m_{4,2^3} & \m_{4,2^3} & \m_{4^2,2} & \m_{6,2^2} & \m_{6,2^2} & \m_{4^2,2} & \m_{8,2}\\
\m_{4,2^3} & \m_{4^2,2} & \m_{4^2,2} & \m_{6,2^2} & \m_{6,4} & \m_{6,4} & \m_{8,2}\\
\m_{4,2^3} & \m_{6,2^2} & \m_{4^2,2} & \m_{4^2,2} & \m_{6,4} & \m_{8,2} & \m_{6,4}\\
\m_{6,2^2} & \m_{6,2^2} & \m_{6,4} & \m_{8,2} & \m_{8,2} & \m_{6,4} & \m_{10}
\end{bmatrix}.
\end{align}
From the second set we get precisely the top $6\times6$ principal submatrix of the above matrix, so it gives no new information, so $\EEE\SS_{10}=\{\langle A_{(1)},\EE_{(1),10}\rangle+\langle A_{(1^2)},\EE_{(1^2),10}\rangle: A_{(1)}\in S_4^+,A_{(1^2)}\in S_3^+\}=\{\langle A,\RR_{10}^e\rangle:A\in S_7^+\}$.
$\triangle$
\end{ex}

\begin{ex}\label{APPsses8=nnes8}
\emph{Proof of Theorem} \ref{teoES8}.

From Example \ref{appes8} we obtain $\EEE\SS_8^*=\{(a,b,c,d,e)\in\R^5:Q\succeq0\}$ where $Q=X\oplus Y\oplus Z$ is the $5\times5$ block diagonal matrix with blocks 
\begin{center}
$X=\begin{bmatrix}
a & b\\
b & c
\end{bmatrix}$,\qquad $Y=\begin{bmatrix}
b-a & d-b\\
d-b & e-c
\end{bmatrix}$,\qquad $Z=d-c$,
\end{center} 

Suppose $(a,b,c,d,e)$ spans a extreme ray of $\EEE\SS_8^*$.
If $\rank X=0$ then $a=b=c=0$ so $Y=
\begin{bmatrix}
0 & d\\ d & e
\end{bmatrix}$, and then $d=0$ for it to be psd. So in this case we only find one extreme ray spanned by $(0,0,0,0,1)$, with kernel 
\begin{align}\label{kernel1}
\{(v,w,x,y,z)\in\R^5:y=0\}
\end{align}
If $\rank Y=0$ then $d=b=a$ and $e=c$, so $X=\begin{bmatrix}
a & a\\
a & c
\end{bmatrix}$ and $c\ge a$ for it to be psd, so $Z=a-c$ must be zero, and then $a=b=c=d=e$. So in this case there is only one extreme ray, spanned by $(1,1,1,1,1)$ and with kernel
\begin{align}\label{kernel2}
\{(v,w,x,y,z)\in\R^5:v=-w\}
\end{align}
If $\rank X=2$ then $\ker Q$ is strictly contained in (\ref{kernel2}) or if $\rank Y=2$ then $\ker Q$ is strictly contained in (\ref{kernel1}), so no extreme rays in these cases.

It only remains to consider the case where $\rank X=\rank Y=1$. Notice that the first rows of these matrices do not depend on $c$, so, since $Z=d-c$, only $c=d$ will possibly lead to a maximal kernel, so $d=c$. From the determinants we have $ac=b^2$ and $(b-a)(e-c)=(c-b)^2$. Hence $a=0$ implies $b=0$ and $c=b$, contradicting that $X$ has rank $1$. So $a\ne0$. If $a=b$ then $a=b=c$ and the kernel of $Q(a,b,c,d,e)$ would be strictly contained in (\ref{kernel2}), so also $a\ne b$. Hence $c=\frac{b^2}a$ and $e-c=\frac{(c-b)^2}{b-a}$, so $e=\frac{b^2}a +\frac{(\frac{b^2}a-b)^2}{b-a}=\frac{b^3}{a^2}$. And in this case we obtain the $3$-dimensional kernel
\begin{align*}
\ker Q(a,b,c,d,e)=\{(v,w,x,y,z)\in\R^5:v=-as,w=bs,x=-at,y=bt,\quad s,t\in\R\}
\end{align*}
which is not contained in (\ref{kernel1}) or (\ref{kernel2}), so it is maximal. For $Q$ to be psd we only need $a>0$ and $b>a$, so if we set $a=1$ and $b=t$ we obtain a family of extreme rays, each one spanned by $(1,t,t^2,t^2,t^3)$ for each $t>1$. Finally notice, that all the extreme rays obtained come from point evaluations: $(0,0,0,0,1)$ is the image of $(0,0,0,1)$ under $\v_4'$, which is clearly in $\HH_8'$, and $(1,t,t^2, t^2, t^3)$ for $t\ge1$ comes from a point evaluation since $\begin{bmatrix}
1 & 1 & t\\
1 & t & t^2\\
t & t^2 & t^3
\end{bmatrix}$ and $\begin{bmatrix}
1 & t\\
t & t^2
\end{bmatrix}$
are positive semidefinite for $t\ge1$.
\end{ex}

\begin{prop}\label{Prop5678}
$\trop(\EEE\PP_{2d}^*)\subsetneq\trop(\EEE\SS_{2d}^*)$ for $d=5,6,7,8$.
\end{prop}
\begin{proof}
We use Algorithm \ref{Algorithm} and also an algorithm to compute $\trop(\EEE\SS_{2d}^*)$ based on Theorem \ref{ThmEvenPartialSym} (although the latter is not necessary to prove the claim). The defining inequalities (supporting hyperplanes of the facets) of $\trop(\EEE\PP_{10}^*)$ are:
\begin{align*}
-z_{2^5}+z_{4,2^3}&\ge0,\\
-z_{4^2,2}+z_{6,2^2}&\ge0,\\
-z_{6,2^2}+z_{6,4}&\ge0,\\  
-z_{6,4}+z_{8,2}&\ge0,\\
 z_{2^5}-2z_{4,2^3}+z_{4^2,2}&\ge0,\\
 z_{4,2^3}-2z_{4^2,2}+z_{6,4}&\ge0,\\
 z_{4,2^3}-2z_{6,2^2}+z_{8,2}&\ge0,\\
 z_{4^2,2}-2z_{6,4}+z_{10}&\ge0,\\     
 z_{6,2^2}-2z_{8,2}+z_{10}&\ge0,\\
 z_{2^5}-3z_{6,2^2}+z_{6,4}+z_{8,2}&\ge0.
\end{align*}

From the $2\times2$ minors obtained from equation (\ref{partialmat10}) of Example \ref{appes10} we get that $\trop(\EEE\SS_{10}^*)$ is given by the same inequalities as above, except for the last one.

The defining inequalities of $\trop(\EEE\PP_{12}^*)$ are
\begin{align*}
-z_{2^6}+z_{4,2^4}&\ge0,\\
-z_{4^2,2^2}+z_{6,2^3}&\ge0,\\
-z_{4^3}+z_{6,4,2}&\ge0,\\ 
-z_{6,2^3}+z_{6,4,2}&\ge0,\\
-z_{6^2}+z_{8,4}&\ge0,\\ 
-z_{8,2^2}+z_{8,4}&\ge0,\\
z_{2^6}-2z_{4,2^4}+z_{4^2,2^2}&\ge0,\\ 
z_{2^6}-2z_{6,2^3}+z_{6^2}&\ge0,\\
z_{4,2^4}-2z_{4^2,2^2}+z_{4^3}&\ge0,\\ 
z_{4,2^4}-2z_{6,2^3}+z_{8,2^2}&\ge0,\\
z_{4^2,2^2}-2z_{6,4,2}+z_{6^2}&\ge0,\\
z_{4^3}-2z_{6,4,2}+z_{8,2^2}&\ge0,\\
z_{6,2^3}-2z_{8,2^2}+z_{10,2}&\ge0,\\
z_{6^2}-2z_{8,4}+z_{10,2}&\ge0,\\ 
z_{8,2^2}-2z_{10,2}+z_{12}&\ge0,\\ 
z_{2^6}-3z_{6,2^3}+z_{6,4,2}+z_{8,2^2}&\ge0,\\
z_{4^2,2^2}+z_{4^3}-3z_{6,4,2}+z_{10,2}&\ge0,\\
z_{4^3}+z_{6,2^3}-3z_{6,4,2}+z_{6^2}&\ge0,\\
z_{4^3}+z_{6^2}-3z_{8,4}+z_{12}&\ge0.
\end{align*}

The last four facets are witnesses of the difference of the limit cones. The first 15 defining inequalities above are also defining inequalities of $\trop(\EEE\SS_{12}^*)$, which has also the additional facet $z_{4^3}-2z_{8,4}+z_{12}\ge0$. 

$\trop(\EEE\PP_{14}^*)$ has 59 facets, $\trop(\EEE\SS_{14}^*)$ has 28 facets (27 in common and the additional facet $z_{4^3,2}-2z_{8,4,2}+z_{12,2}\ge0$). An example of a facet of $\trop(\EEE\PP_{14}^*)$ which is not a facet of $\trop(\EEE\SS_{14}^*)$ is $z_{8,4,2}+z_{8,6}-3z_{10,4}+z_{14}\ge0$.  

$\trop(\EEE\PP_{16}^*)$ has 165 facets of which 118 are of not of the form $e_i-e_j\ge0$ or $e_i-2e_j+e_k\ge0$. An example of a facet of $\trop(\EEE\PP_{16}^*)$ which is not a facet of $\trop(\EEE\SS_{16}^*)$ is $z_{10,4,2}+z_{10,6}-3z_{12,4}+z_{16}\ge0$
\end{proof}

\end{document}